\documentclass[12pt, a4paper]{amsart}
\usepackage{amsfonts, amssymb, amsmath}
\usepackage{amsthm}
\usepackage{tabularx}
\usepackage[margin=2.5cm]{geometry}
\usepackage{enumerate}
\usepackage{scrextend}
\usepackage[all]{xy}
\usepackage[vlined, ruled, boxed]{algorithm2e}

\newcommand{\N}{{\mathbb N}}

\newcommand{\Q}{{\mathbb Q}}

\newcommand{\Z}{{\mathbb Z}}
\newcommand{\Pp}{{\mathbb P}}
\newcommand{\A}{{\mathbb A}}
\newcommand{\Cc}{{\mathbb C}}
\newcommand{\Oo}{\mathcal{O}}
\newcommand{\B}{\mathcal{B}}

\newcommand{\se}[2]{\left\lbrace #1 \mbox{ }\vline\mbox{ } #2 \right\rbrace}

\newcommand{\tl}[1]{\tilde{#1}}
\newcommand{\st}{^{\ast}}
\newcommand{\ts}{_{\ast}}
\newcommand{\rd}[1]{\lfloor #1\rfloor}

\newtheorem{thm}{Theorem}[section]
\newtheorem{pro}[thm]{Proposition}
\newtheorem{cor}[thm]{Corollary}
\newtheorem{lem}[thm]{Lemma}
\theoremstyle{definition}
\newtheorem{rk}[thm]{Remark}
\newtheorem{eg}[thm]{Example}

\newtheorem*{defn}{Definition}

\begin{document}

\title{Threefolds of Kodaira dimension one}
\author{Hsin-Ku Chen}

\begin{abstract}
We prove that for any smooth complex projective threefold of Kodaira dimension one,
the $m$-th pluricanonical map is birational to the Iitaka fibration for every $m\geq5868$ and divisible by $12$.
\end{abstract}
\maketitle
\section{Introduction}
By the result of Hacon-M$^{\mbox{\small{c}}}$Kernan \cite{hm}, Takayama \cite{ta} and Tsuji \cite{ts}, it is known that
for any positive integer $n$ there
exists an integer $r_n$ such that if $X$ is an $n$-dimensional smooth complex projective variety of general type, then $|rK_X|$ 
defines a birational morphism for all $r\geq r_n$.
It is conjectured in \cite{hm} that a similar phenomenon occurs for any projective variety of non-negative Kodaira dimension. That is,
for any positive integer $n$ there exists a constant $s_n$ such that, if $X$ is an $n$-dimensional smooth projective variety
of non-negative Kodaira dimension and
$s\geq s_n$ is sufficiently divisible, then the $s$-th pluricanonical map of $X$ is birational to the Iitaka fibration.\par
We list some known results related to this problem.
In 1986, Kawamata \cite{k2} proved that there is an integer $m_0$ such that for any terminal threefold $X$ with
Kodaira dimension zero, the $m_0$-th plurigenera of $X$ is non-zero. Later on, Morrison proved that one can take
$m_0=2^5\times3^3\times5^2\times7\times11\times13\times17\times19$. See \cite{m} for details.
In 2000, Fujino and Mori \cite{fm} proved that if $X$ is a smooth projective variety with Kodaira dimension one
and $F$ is a general fiber of the Iitaka fibration of $X$, then there exists a integer $M$, which depends on the dimension of $X$,
the middle Betti number of some finite covering of $F$ and the smallest integer so that the pluricanonical system of $F$
is non-empty, such that the $M$-th pluricanonical map of $X$ is birational to the
Iitaka fibration. Viehweg and D-Q Zhang \cite{vzh} proved an analog result for the Kodaira dimension two case.
Recently, Birkar and D-Q Zhang \cite{bz} proved that a Fujino-Mori type statement holds for any non-negative Kodaira dimension.
Note that if $C$ is a curve of Kodaira dimension zero, then $|K_C|$ is non-empty and $b_1(C)=2$. Also if $S$ is a surface of
Kodaira dimension zero, then $|12K_S|$ is non-empty and $b_2(S)\leq22$. Thus the Hacon-M$^{\mbox{\small{c}}}$Kernan conjecture holds
for varieties with dimension less than or equal to three.\par
It is also interesting to find an explicit value to bound the Iitaka fibration. In dimension one, it is well-known that the third-pluricanonical
map is the Iitaka fibration. For the surfaces case, Iitaka \cite{i} proved that the $m$-th pluricanonical system is birational to the Iitaka
fibration if $m\geq86$ and divisible by $12$.
For threefolds of general type, J. A. Chen and M. Chen \cite{cm} proved that the $m$-th pluricanonical map is birational if $m\geq61$. For
threefolds of Kodaira dimension two, Ringler \cite{r} proved that the $m$-th pluricanonical map is birational to the Iitaka fibration if
$m\geq48$ and divisible by $12$. In this article we will study threefolds of Kodaira dimension one. We prove the following.
\begin{thm}\label{mt}
	Let $X$ be a smooth complex projective threefold of Kodaira dimension one. Then 
	the $m$-th pluricanonical map of $X$ is birational to the Iitaka fibration for every $m\geq5868$ and divisible by $12$.\par
	More precisely, let $F$ be a general fiber of the Iitaka fibration of $X$, we have
	\begin{enumerate}[(1)]
	\item If $F$ is birational to a K3 surface, then $|mK_X|$ defines the Iitaka fibration if $m\geq86$.
	\item If $F$ is birational to an Enriques surface, then $|mK_X|$ defines the Iitaka fibration if $m$ is an even integer which is
		greater than or equal to $42$.
	\item If $F$ is birational to an abelian surface or a bielliptic surface, then $|mK_X|$ defines the Iitaka fibration if $m\geq5858$
		(and divisible by $12$ when $F$ is bielliptic).
	\end{enumerate}
\end{thm}
Combining the work of J. A. Chen-M. Chen \cite{cm} and Ringler \cite{r},
we have the following effective bound for threefolds of positive Kodaira dimension.
\begin{cor}
	Let $X$ be a smooth complex projective threefold of positive Kodaira dimension.
	Then $|mK_X|$ defines the Iitaka fibration if $m\geq5868$ and divisible by $12$.
\end{cor}

We now give a rough idea of the proof of Theorem \ref{mt}. 
If the Iitaka fibration maps to a non-rational curve, then the boundedness
of the Iitaka fibration can be easily derived using weak-positivity.
Now assume that the Iitaka fibration of $X$ maps to a rational curve.
We may assume that $X$ is minimal and hence the general fiber of the Iitaka fibration
is a K3 surface, an Enriques surface, an abelian surface or a bielliptic surface.
If the general fiber has non-zero Euler characteristic, i.e., if the 
Iitaka fibration is a K3 or an Enriques fibration, we observe the following fact. One may write $K_X$ as a pull-back of an ample $\Q$-divisor.
The degree of this $\Q$-divisor is determined by the singularities of $X$. If the degree is large,
then a small multiple of $K_X$ defines the Iitaka fibration. Assume that the degree
is small, then the singularities of $X$ are bounded: the local index of singular points of $X$ can not be too large,
and the total number of singular points is bounded. This implies the degree has a lower bound.
With the help of a computer, we get a good estimate of this lower bound and hence a good effective bound for the Iitaka fibration.\par
If the Iitaka fibration is an abelian or a bielliptic fibration, then above techniques do not work. Our main tool to solve the boundedness
problem is Fujino-Mori's canonical bundle formula. Since the second Betti number of an abelian surface is not too large,
the canonical bundle formula behaves well in this situation. One can show that the Euler characteristic of the
canonical cover of the given threefold is zero, hence it has non-trivial irregularity or non-trivial geometric genus. We can first
solve the boundedness problem of the Iitaka fibration for irregular threefolds and threefolds with non-trivial geometric genus,
and then compare terms in the canonical bundle formula between the original threefold and its canonical cover.\par
One can always replace our smooth threefold by its minimal model, hence throughout this article, we always assume that our threefold is minimal
and with terminal singularities. Since the abundance conjecture is known to be true in dimension three, the Iitaka fibration is a morphism.
We will denote it by $f\colon X\rightarrow C$ and hence $K_X$ is a pull-back of some ample $\Q$-divisor on $C$. If $C$ is not rational,
we will prove the desired boundedness in Section \ref{sp} (cf. Proposition \ref{nonr}). In the later sections we will always assume that
$C$ is a rational curve. We discuss K3/Enriques fibrations in Section \ref{sk3} and abelian/bielliptic fibrations in Section \ref{sabl}.
We will prove Theorem \ref{mt} in Section \ref{sk1}, which is a collection of the result in previous sections. We also compute several examples,
which are threefolds of Kodaira dimension one such that a small pluricanonical system do not define the Iitaka fibration.\par

I thank Christopher Hacon for suggesting this question to me and giving me lots of useful suggestions. I would like to thank Jungkai Alfred Chen
for discussing this question with me and giving me lots of helpful comments. 
The author was supported by the Ministry of Science and Technology, Taiwan and NCTS. This work was done while the author was visiting the
University of Utah. The author would like to thank the University of Utah for its hospitality.

\section{Preliminary}\label{sp}
\subsection{The canonical bundle formula}\label{scbf}
Let $X$ be a minimal terminal threefold of Kodaira dimension one. Since the abundance conjecture holds for threefolds, $K_X$ is semi-ample.
Hence the Iitaka fibration $f\colon X\rightarrow C$ is a morphism and $K_X$ is the pull-back of an ample divisor on $C$.
We denote a general fiber of $X\rightarrow C$ by $F$. By \cite{fm} we have the following canonical bundle formula
\begin{thm}[\cite{fm}]\label{fmcbf}
	Let $X$ be a minimal terminal threefold of Kodaira dimension one. Let $f\colon X\rightarrow C$ be the Iitaka fibration.
	Let $F$ be a general fiber of $f$ and let $b$ be the smallest integer such that $|bK_F|$ is non-empty.
	Then \[bK_X=f\st(b(K_C+M+B)),\] where $M$ and $B$ are $\Q$-divisors on $C$ such that
	\begin{enumerate}[(1)]
		\item For all $P\in C$, let $n_P$ be the smallest integer such that $n_PM$ is Cartier at $P$.
			Then $\phi(n_P)\leq \dim_{\Cc}H^2_{prim}(\tl{F},\Cc)$. Here $\phi$ is Euler's function and $\tl{F}$ is the canonical cover of $F$.
		\item $B=\sum_{P\in C} s_PP$, where $s_P=1-\frac{v_P}{bn_pu_p}$ for some $u_p, v_p\in\N$
			satisfying $1\leq v_p\leq bn_p$.
		\item $u_P$ satisfies the following condition: $\lambda f\st P$ is integral if and only if $\lambda\in\frac{1}{u_P}\Z$.
	\end{enumerate}
\end{thm}
\begin{proof}
	Since $K_X$ is semi-ample, the exceptional term in the formula in \cite[Proposition 2.2]{fm} is zero. Hence we have
	$bK_X=f\st(b(K_C+M+B))$. (1) follows from \cite[3.6, 3.8]{fm} and (2) follows form the proof of \cite[Theorem 4.5]{fm}.
	(3) follows form the observation that $f\st(bn_PB)=bn_PK_{X/C}-f\st(bn_P M)$ is an integral Weil divisor over $P$.
\end{proof}

We will write $A=K_C+M+B$ for convenience.\par
\begin{lem}\label{cbf} 
	We have $f\ts\Oo_X(rbK_X)=\Oo_C(\lfloor rbA\rfloor)$ for all integers $r\geq0$.
	In particular, if $C\cong\Pp^1$ and $h^0(X,rbK_X)\geq2$, then $|rbK_X|$ defines the Iitaka fibration.
\end{lem}
\begin{proof}
	By \cite[Proposition 2.2]{fm} and the projection formula we have $f\ts\Oo_X(rbK_X)^{\ast\ast}=\Oo_C(\lfloor rbA\rfloor)$.
	Since $f\ts\Oo_X(rbK_X)$ is torsion free, it is locally free, hence $f\ts\Oo_X(rbK_X)^{\ast\ast}=f\ts\Oo_X(rbK_X)$ and we have
	\[f\ts\Oo_X(rbK_X)=\Oo_C(\lfloor rbA\rfloor).\]
	Now if $C\cong \Pp^1$ and $H^0(X,rbK_X)\geq 2$, then $\lfloor rbA\rfloor$ has positive degree, hence very ample.
	Thus $H^0(X,\Oo_X(rbK_X))=H^0(X,f\ts\Oo_X(rbK_X))=H^0(C,\Oo_C(\rd{rbA})$ defines the morphism $X\rightarrow C$.
\end{proof}

\subsubsection{Canonical bundle formula for two-dimensional elliptic fibrations}
Let $S$ be a smooth surface of Kodaira dimension one and let $f\colon S\rightarrow C$ be the Iitaka fibration, which is an elliptic fibration.
We introduce the connection between Kodaira's canonical bundle formula for elliptic fibrations and Fujino-Mori's formula we mentioned above.\par
Let $U\subset C$ be the smooth locus of $S\rightarrow C$. We can define a morphism $J_U\colon U\rightarrow\Cc$ which is induced by the $j$-function
of $f^{-1}(P)$ for all $P\in U$ (see \cite[Section 2]{fu} for the detail). Since $C$ is a curve, $J_U$ induces a morphism
$J\colon C\rightarrow \Pp^1$. 
\begin{lem}\label{efcbf}
	Notation as above. Write $K_S=f\st(K_C+M+B)$ as in Theorem \ref{fmcbf}.
	\begin{enumerate}[(1)]
	\item $12M=J\st\Oo_{\Pp^1}(1)$.
	\item Let $Z_0=\se{P\in C}{f^{-1}(P)\mbox{ is a multiple fiber}}$ and $B_0=B|_{Z_0}$, $B_1=B-B_0$. Then
		$M+B_1$ is integral and \[\deg (M+B_1)=\chi(\Oo_S).\]
	\end{enumerate}
\end{lem}
\begin{proof}
	One can verify (1) by computing the coefficients of $B$ (they are log canonical thresholds by \cite[Proposition 4.7]{fm})
	via Kodaira's explicit classification, and applying \cite[(2.9)]{fu}. (2) follows form \cite[Theorem 12]{kod}.
\end{proof}

\subsection{Koll\'{a}r vanishing theorem}
\begin{thm}[\cite{ko}, Theorem 10.19]\label{ko}
	Let $f\colon X\rightarrow Y$ be a surjective morphism between normal and proper varieties. Let $N$, $N'$ be rank $1$, reflexive,
	torsion-free sheaves on $X$. Assume that $N\equiv K_X+\Delta+f\st M$, where $M$ is a $\Q$-Cartier $\Q$-divisor on $Y$ and
	$(X,\Delta)$ is klt. Then
	\begin{enumerate}[(1)]
	\item $R^jf\ts N$ is torsion free for $j\geq0$.
	\item Assume that in addition that $M$ is nef and big. Then
		\[H^i(Y,R^jf\ts N)=0 \quad \mbox{for }i>0,j\geq0.\]
	\item Assume that $M$ is nef and big and let $D$ be any effective Weil divisor on $X$ such that $f(D)\neq Y$. Then
		\[H^j(X,N)\rightarrow H^j(X,N(D))\quad\mbox{is injective for }j\geq0.\]
	\item If $f$ is generically finite and $N'\equiv K_X+\Delta+F$ where $F$ is a $\Q$-Cartier $\Q$-divisor on $X$ which is $f$-nef, then
		$R^jf\ts N'=0$ for all $j>0$.
	\end{enumerate}
\end{thm} 
\begin{rk}\label{lfree}
	Under our assumption that $X$ is a minimal terminal threefold of Kodaira dimension one and $f\colon X\rightarrow C$ is the Iitaka fibration,
	Theorem \ref{ko} (1) implies that $R^jf\ts(mK_X)$ is locally free because any torsion-free sheaf
	on a curve is locally free.
\end{rk}
\begin{pro}[\cite{ko1},Proposition 7.6]\label{rkk}
	Let $X$, $Y$ be smooth projective varieties, $\dim X=n$, $\dim Y=k$, and let $\pi\colon X\rightarrow Y$ be a surjective map with
	connected fibers. Then $R^{n-k}\pi\ts \omega_X=\omega_Y$.
\end{pro}
\begin{cor}\label{rk}
	The same conclusion of Proposition \ref{rkk} holds if $X$ has canonical singularities.
\end{cor}
\begin{proof}
	Let $\phi\colon \tl{X}\rightarrow X$ be a resolution of $X$ and $\tl{\pi}\colon \tl{X}\rightarrow Y$ be the composition of $\pi$ and $\phi$.
	By the Grothendieck spectral sequence, we have
	\[E^{p,q}_2=R^p\pi\ts(R^q\phi\ts \omega_{\tl{X}})\Rightarrow R^{p+q}\tl{\pi}\ts\omega_{\tl{X}}.\]
	By Grauert-Riemenschneider vanishing (or Theorem \ref{ko} (4) above) we have $R^q\phi\ts\omega_{\tl{X}}=0$ for all $q>0$,
	hence \[\omega_Y=R^{n-k}\tl{\pi}\ts\omega_{\tl{X}}=R^{n-k}\pi\ts(\phi\ts\omega_{\tl{X}})=R^{n-k}\pi\ts\omega_X.\]
\end{proof}

Now assume that $X$ is a minimal terminal threefold of Kodiara dimension one and $X\rightarrow C$ is the Iitaka fibration.
Let $F$ be a general fiber. We may write $K_X=f\st A$. Let $\lambda=1/\deg A$, hence we have $F\equiv\lambda K_X$.
\begin{lem}\label{deg} $h^0(X,mK_X)\geq r$ if $m>\lambda r+1$ and $|mK_F|$ is non-empty.
\end{lem}
\begin{proof}
	Choose $r$ general fibers $F_1$, ..., $F_r$. Consider the exact sequence
	\[0\rightarrow\Oo_X(mK_X-F_1-...-F_r)\rightarrow\Oo_X(mK_X)\rightarrow\bigoplus_{i=1}^r\Oo_{F_i}(mK_X)
		=\bigoplus_{i=1}^r\Oo_{F_i}\rightarrow0.\]
	Note that $mK_X-F_1-...-F_r\equiv K_X+(m-1-\lambda r)K_X$. By our assumption we have $m-1-\lambda r>0$, hence
	\[H^1(X,\Oo_X(mK_X-F_1-...-F_r))\rightarrow H^1(X,\Oo_X(mK_X))\]
	is injective by Theorem \ref{ko} (3) with $M=\frac{m-1-\lambda r}{\lambda}P$, $D=F_1+...+F_r$ and $\Delta=0$,
	here $P$ is a general point on $C$. Hence \[H^0(X,\Oo_X(mK_X))\rightarrow \bigoplus_{i=1}^rH^0(F_i,\Oo_{F_i})\]
	is surjective and so $h^0(X,\Oo_X(mK_X))\geq r$.
\end{proof}

\subsection{The singular Riemann-Roch formula}\label{srr}
When studying terminal threefolds, a basic tool is Reid's singular Riemann-Roch formula \cite{re}:
\begin{equation}\label{eqrr}
\begin{split}
	\chi(\Oo_X(D))=\chi(\Oo_X)&+\frac{1}{12}D(D-K_X)(2D-K_X)+\frac{1}{12}D.c_2(X)\\&+
	\sum_{P\in\B(X)}\left(-i_P\frac{r_P^2-1}{12r_P}+\sum_{j=1}^{i_P-1}\frac{\overline{jb_P}(r_P-\overline{jb_P})}{2r_P}\right),
\end{split}
\end{equation}

where $\B(X)=\{(r_P,b_P)\}$ is the basket data of $X$ (see \cite[Section 10]{re} for a more precise definition) and $i_P$
is the smallest integer such that $\Oo_X(D)\cong\Oo_X(i_PK_X)$ near $P$.\par
Take $D=K_X$, one has
\begin{equation}\label{eqc2} K_X.c_2(X)=-24\chi(\Oo_X)+\sum_{P\in\B(X)}\left(r_P-\frac{1}{r_P}\right).\end{equation}
Now take $D=mK_X$ and replace $K_X.c_2(X)$ by $\chi(\Oo_X)$ and the contribution of singularities, we get
the following plurigenus formula \cite[Section 2]{ch}:
\[\chi(mK_X)=\frac{1}{12}m(m-1)(2m-1)K_X^3+(1-2m)\chi(\Oo_X)+l(m),\] here
	\[l(m)=\sum_{P\in\B(X)}\sum_{j=1}^{m-1}\frac{\overline{jb_P}(r_P-\overline{jb_P})}{2r_P}.\]
If one assumes that $X$ is minimal and of Kodaira dimension one, then $K^3_X=0$ and one has
\begin{equation}\label{e2}
	\chi(mK_X)=(1-2m)\chi(\Oo_X)+l(m).
\end{equation}
\begin{lem}\label{chi}
	Assume that $X$ is minimal and of Kodaira dimension one and $F$ is a general fiber of the Iitaka fibration of $X$.
	If $F\equiv \lambda K_X$, then \[\sum_{P\in\B(X)}\left(r_P-\frac{1}{r_P}\right)=24\chi(\Oo_X)+\frac{12}{\lambda}\chi(\Oo_F).\]
\end{lem}
\begin{proof}
	Consider the exact sequence \[0\rightarrow\Oo_X(K_X)\rightarrow\Oo_X(K_X+F)\rightarrow\Oo_F(K_F)\rightarrow0.\]
	We have \[\chi(\Oo_F)=\chi(K_F)=\chi(K_X+F)-\chi(K_X)=\frac{1}{12}F.c_2(X)=\frac{\lambda}{12}K_X.c_2(X),\]
	where the third equation follows from equation (\ref{eqrr}).
	Thus \[\frac{12}{\lambda}\chi(\Oo_F)=K_X.c_2(X)=-24\chi(\Oo_X)+\sum_{P\in\B(X)}\left(r_P-\frac{1}{r_P}\right)\]
	by equation (\ref{eqc2}). This proves our lemma.
\end{proof}
\subsection{Surfaces of Kodaira dimension zero}
Let $X$ be a threefold of Kodaira dimension one and let $X\rightarrow C$ be the Iitaka fibration. Then a general fiber of $X\rightarrow C$
is a surface of Kodaira dimension zero. We will briefly introduce the classification of surfaces of Kodaira dimension zero. 
For more detail, see \cite[Section VIII]{b}.\par
Let $S$ be a minimal surface of Kodaira dimension zero. The $S$ belongs to one of the following classes:
\begin{enumerate}
\item $h^0(S,K_S)=1$, $h^1(S,K_S)=0$. We have $K_S\sim 0$. $S$ is called a K3 surface.
\item $h^0(S,K_S)=h^1(S,K_S)=0$. We have $K_S\not \sim 0$ but $2K_S\sim 0$. $S$ is a quotient of a K3 surface by a fixed-point free involution.
	In this case $S$ is called an Enriques surface.
\item $h^0(S,K_S)=1$, $h^1(S,K_S)=2$. We have $K_S\sim 0$ and $S$ is an abelian surface.
\item $h^0(S,K_S)=0$, $h^1(S,K_S)=1$. We have $K_S\not\sim0$ but either $2K_S$, $3K_S$, $4K_S$ or $6K_S\sim 0$. $S$ is a quotient of a
	product of two elliptic curves by a finite group, and is called a bielliptic surface.
\end{enumerate}
In conclusion, we know that $12K_S\sim 0$ for all surfaces of Kodaira dimension zero.

\subsection{Weak positivity}

\begin{thm}[\cite{v}, Theorem III]
	Let $g\colon T\rightarrow W$ be any surjective morphism between non-singular projective varieties. Then $g\ts\omega_{T/W}^k$
	is weakly positive for any $k>0$.
\end{thm}

\begin{rk}
	When $T\rightarrow W$ is the Iitaka fibration of $T$ and $W$ is a curve, we have $g\ts\omega^k_{T/W}$ is a line bundle by remark \ref{lfree}.
	By \cite[Section 5]{k3}, $g\ts\omega^k_{T/W}$ is pseudo-effective, which is equivalent to say that $\deg g\ts\omega^k_{T/W}\geq0$.\par
	Furthermore, the same conclusion holds if $T$ has canonical singularities because if $\phi\colon \tl{T}\rightarrow T$ is
	a resolution of singularity, then $\phi\ts\omega^k_{\tl{T}}=\omega^k_T$.\par
\end{rk}
\begin{pro}\label{nonr}
	Let $X$ be a minimal terminal threefold with Kodaira dimension one and let $f\colon X\rightarrow C$ be the Iitaka fibration of $X$.
	Assume that $g(C)\geq1$.
	Let $F$ be a general fiber and let $b$ be an integer such that $|bK_F|$ is non-empty and $b\geq2$. 
	Then $|bK_X|$ is non-empty and $|3bK_X|$ defines the Iitaka fibration.\par
	In particular, $|mK_X|$ defines the Iitaka fibration for all $m\geq 24$ and divisible by $12$.
\end{pro}
\begin{proof}
	First assume that $g(C)\geq2$. By weak-positivity, \[\deg f\ts(bK_X)\geq\deg bK_C=b(2g(C)-2).\] We have
	\begin{align*}
	H^0(X,bK_X)=H^0(C,f\ts(bK_X))&\geq\chi(f\ts(bK_X))\\&=\deg f\ts(bK_X)+1-g(C)\geq(2b-1)(g(C)-1)>0.
	\end{align*}
	Moreover, $\deg f\ts(2bK_X)\geq 8g(C)-8>2g(C)+1$ under our assumption that $b\geq2$, hence $f\ts(2bK_X)$ is very-ample,
	which implies $|2bK_X|$ defines the Iitaka fibration. Since $H^0(X,bK_X)\neq0$, $|3bK_X|$ also defines the Iitaka fibration.\par
	Now assume that $g(C)=1$. One has $\deg f\ts(bK_X)\geq0$ by the weak positivity. Moreover $H^1(C,f\ts(bK_X)\otimes P)=0$ for all
	$P\in Pic^0(C)$ by Theorem \ref{ko} (2). Hence $\deg f\ts(bK_X)$ is positive since otherwise after taking $P=(f\ts(bK_X))\st$
	we get $H^1(C,\Oo_C)=0$, which is a contradiction. Thus $h^0(bK_X)\neq0$.
	Now we have $\deg f\ts(3bK_X)\geq3$ since one has the natural inclusion	$f\ts(bK_X)^{\otimes3}\rightarrow f\ts(3bK_X)$.
	The conclusion is that $f\ts(3bK_X)$ is very-ample, so that $|3bK_X|$ defines the Iitaka fibration.\par
	Note that $|mK_X|$ defines the Iitaka fibration if $m\geq 3b$ and $m$ divisible by $b$. We know that
	$b\in\{2,3,4,6\}$. It is easy to see that $|mK_X|$ defines the Iitaka fibration for all $m\geq 24$ and divisible by $12$.
\end{proof}
\subsection{Canonical covers}
Let $X$ be a terminal threefold. By \cite[Corollary 1.9]{re2}, there exists a ramified cover $\mu\colon \tl{X}\rightarrow X$,
so called the canonical cover of $X$, such that $\tl{X}$ is terminal Gorenstein and $\mu$ is \'{e}tale in codimension one.\par
We summarize some properties about the canonical cover when $X$ is terminal of Kodaira dimension one.
\begin{lem}\label{cv}
	Let $X$ be a projective terminal threefold of Kodaira dimension one and let $\mu\colon \tl{X}\rightarrow X$ be the canonical cover.
	\begin{enumerate}[(1)]
	\item $\tl{X}$ is also a projective terminal threefold of Kodaira dimension one.
	\item Let $f\colon X\rightarrow C$ (resp. $\tl{f}\colon \tl{X}\rightarrow \tl{C}$) be the Iitaka fibration and
		assume that $K_X=f\st A$ (resp. $K_{\tl{X}}=\tl{f}\st\tl{A}$). Then $\mu\circ f$ factorize through a morphism $\nu:\tl{C}\rightarrow C$.
		We have $\tl{A}=\nu\st A$.
	\item If we write $A=K_C+M+B$ and $\tl{A}=K_{\tl{C}}+\tl{M}+\tl{B}$. Then $\tl{M}=\nu\st M$.
	\end{enumerate}
\end{lem} 
\begin{proof}
	Let $r$ be the global Cartier index of $K_X$ and let $U$ be an affine chart such that $\omega_X^{[r]}(U)$ is free. One can define
	$\Oo_{\tl{X}}(U)=\bigoplus_{i=0}^{r-1} \omega_X^{[i]}(U)$ where the ring structure is induced by the isomorphism
	$\omega_X^{[r]}(U)\cong\Oo_X(U)$. One can glue those affine schemes together and get a Noetherian scheme $\tl{X}$.
	Note that $\tl{X}\rightarrow X$ is proper since since locally the morphism looks like $\Oo_X(U)\rightarrow \Oo_X(U)[t]/(t^r-\sigma)$,
	for some local section $\sigma\in\omega_X^{[r]}(U)$. Hence the pull-back of an ample divisor on $X$ is an ample divisor on
	$\tl{X}$ by the Nakai-Moishezon criterion. Thus $\tl{X}$ is a projective variety. Since $\mu\colon \tl{X}\rightarrow X$ is \'{e}tale in
	codimension one, $K_{\tl{X}}=\mu\st K_X$. We have $K_{\tl{X}}$ is nef and the numerical dimension of $K_{\tl{X}}$ and $K_X$ are the same.
	Hence $\tl{X}$ is minimal of Kodaira dimension one.\par
	(2) follows from the fact that there is an embedding $|mK_X|\hookrightarrow |mK_{\tl{X}}|$ for all $m\in\N$.
	(3) follows from \cite[Proposition 5.5]{a}. 
\end{proof}

\section{K3 or Enriques fibrations}\label{sk3}
Let $X$ be a minimal terminal threefold with Kodaira dimension one and \[f\colon X\rightarrow C\cong\Pp^1\] be the Iitaka fibration.
Let $F$ be the general fiber of $X\rightarrow C$ and assume that $F$ is a K3 surface or an Enriques surface.
We have $H^1(F,K_F)=0$. Thus $R^1f\ts(K_X)=0$ because it is locally free by Remark \ref{lfree}. 
Note that \begin{align*}
h^1(X,\Oo_X)=h^2(X,\Oo_X(K_X))&=h^0(C,R^2f\ts\Oo_X(K_X))+h^1(C,R^1f\ts\Oo_X(K_X))\\&=h^0(C,\Oo_C(K_C))=0.
\end{align*}
Here $h^0(C,R^2f\ts\Oo_X(K_X))=h^0(C,\Oo_C(K_C))$ follows from Corollary \ref{rk}. We also have \begin{align*}
h^2(X,\Oo_X)=h^1(X,\Oo_X(K_X))&=h^0(C,R^1f\ts\Oo_X(K_X))+h^1(C,f\ts\Oo_X(K_X))\\&=h^1(C,f\ts\Oo_X(K_X)). \end{align*}\par
If $F$ is an Enriques surface, then $h^0(F,K_F)=0$, hence $f\ts\Oo_X(K_X)=0$ since it is locally free by Remark \ref{lfree}.
We have $h^2(X,\Oo_X)=0$ and \[h^3(X,\Oo_X)=h^0(X,\Oo_X(K_X))=h^0(C,f\ts\Oo_X(K_X))=0.\] Thus $\chi(\Oo_X)=1$.
When $F$ is a K3 surface by the weak positivity we have \[\deg f\ts\Oo_X(K_X)\geq \Oo_C(K_C)=-2,\]
hence $h^1(C,f\ts\Oo_X(K_X))\leq1$.
Thus $h^2(\Oo_X)\leq1$ which implies $\chi(\Oo_X)\leq2$. If $\chi(\Oo_X)<0$, then $h^0(X,K_X)\geq2$
and $|K_X|$ defines the Iitaka fibration by Lemma \ref{cbf}. From now on we assume that $0\leq\chi(\Oo_X)\leq 2$.\par
Recall that we write $K_X=f\st A$ for some $\Q$-divisor $A$ on $C$.
\begin{pro}\label{k3}
	If $f\colon X\rightarrow C$ is a K3 fibration, then $\deg A\geq\frac{1}{42}$.
	If $f\colon X\rightarrow C$ is an Enriques fibration, we have $\deg A\geq\frac{1}{20}$.
	In particular, $|mK_X|$ defines the Iitaka fibration if $m\geq86$ (resp. $m$ is even and $\geq42$)
	if $X\rightarrow C$ is a K3 (resp. an Enriques) fibration.
\end{pro}
\begin{proof}
	Let $\lambda=1/\deg A$. Assume that $\lambda>N$ for some integer $N$, then Lemma \ref{chi} implies 
	\[24\chi(\Oo_X)<\sum_{P\in\B(X)}\left(r_P-\frac{1}{r_P}\right)<24\chi(\Oo_X)+\frac{12}{N}\chi(\Oo_F).\]
	This tells us that $r_P\leq 24\chi(\Oo_X)+\frac{12}{N}\chi(\Oo_F)$ for all $P$ and there is at most
	\[\frac{2}{3}\left(24\chi(\Oo_X)+\frac{12}{N}\chi(\Oo_F)\right)\] non-Gorenstein points on $X$
	since $r-\frac{1}{r}\geq\frac{3}{2}$ for all integer $r>1$. 
	Note that we assume that $0\leq\chi(\Oo_X)\leq2$, hence there are only finitely many possible basket data.
	By Lemma \ref{chi} a basket data corresponds to a unique $\lambda$ once $\chi(\Oo_X)$ is fixed.
	This tells us that $\lambda$ has an upper bound.\par
	Note that the basket data satisfies more conditions.
	We have \[h^1(X,mK_X)=0\] since $h^1(C,f\ts(mK_X))=0$ by Theorem \ref{ko} (2) and $R^1f\ts(mK_X)$ is a zero sheaf because $h^1(F,mK_F)=0$.
	Also \[h^3(X,mK_X)=h^0(X,(1-m)K_X)=0\] for all $m>1$ since $K_X$ is pseudo-effective and not numerically trivial. Thus we have
	$\chi(mK_X)\geq0$ and so equation (\ref{e2}) in Section \ref{srr} yields that \begin{equation}\label{e3}
	(1-2m)\chi(\Oo_X)+l(m)\geq 0,\end{equation}
	where \[l(m)=\sum_{P\in\B(X)}\sum_{j=1}^{m-1}\frac{\overline{jb_P}(r_P-\overline{jb_P})}{2r_P},\]
	for all $m>1$.
	As we have seen before, for a fixed integer $N$ and assuming $\lambda>N$, there are only finitely many possible basket data of $X$.
	Using a computer, one can write down all possible basket data and check whether such basket satisfies (\ref{e3}) or not.
	In the K3 fibration case if one take $N=43$, then there is no basket data satisfying (\ref{e3}) for all $m>1$,
	hence $\lambda\leq42$.
	In the Enriques fibration case using the same technique one can prove that $\lambda\leq 20$.\par
	Finally the boundedness of the Iitaka fibration follows from Lemma \ref{deg} and Lemma \ref{cbf}.
\end{proof}
\begin{rk}
	We remark that the worst possible basket data occurs when $X\rightarrow C$ is a K3 fibration, $\chi(\Oo_X)=2$ and the basket data is
	\[ \{(2,1)\times8,(3,1)\times6,(7,1),(7,2),(7,3)\}\]
	with $\lambda=42$. This kind of fibration exists. See Example \ref{e42}.
\end{rk}
\begin{rk}
	The algorithm for the program in the proof of Proposition \ref{k3} will be given in the appendix.
\end{rk}
\section{Abelian or bielliptic fibrations}\label{sabl}
Let $X$ be a terminal minimal threefold of Kodaira dimension one and assume that $f\colon X\rightarrow C\cong\Pp^1$ is the Iitaka fibration.
Let $F$ be a general fiber of $X\rightarrow C$ and assume that $F$ is either an abelian surface or a bielliptic surface.
We write $K_X=f\st A=f\st(K_C+M+B)$ as in Theorem \ref{fmcbf} and we are going to estimate $\deg A$. 
Recall that for all $P\in C$, let $n_P$ be the smallest integer such that $n_PM$ is Cartier at $P$. Then
$\phi(n_P)\leq \dim_{\Cc}H^2_{prim}(\tl{F},\Cc)=5$. Here $\tl{F}$ is the canonical cover of $F$, which is an abelian surface.
One conclude that $n_P\leq 12$.

\subsection{Abelian or bielliptic fibrations with irregularity greater than one}
\begin{pro}\label{qqq}
	Assume that $h^1(X,\Oo_X)>1$, then $F$ is an abelian surface and $X\rightarrow C$ is isotrivial.
	We have $\deg A\geq\frac{1}{42}$.
\end{pro}
To prove the proposition we need the following criterion of isotriviality due to Campana and Peternell.
\begin{defn} Let $W$ be a smooth threefold of Kodaira dimension one and let $\omega$ be a two-form on $W$.
	Let $g\colon W\rightarrow C$ be the Iitaka fibration.
	We say $\omega$ is a vertical two-form with respect to $g$ if $\omega$ corresponds to an element $s$ such that
	\[ s\in H^0(W,T_{W/C}\otimes K_W)\subset H^0(W,T_W\otimes K_W)\cong H^0(W,\Omega^2_W).\]
	Here $T_W$ denotes the tangent bundle of $W$ and $T_{W/C}$ is the kernel of $T_W\rightarrow g\st T_C$.
\end{defn}
\begin{thm}[\cite{cp}, Theorem 4.2]\label{cp}
	Let $W$ be a smooth projective threefold of Kodaira dimension one and let $\omega$ be a two-form on $W$. Assume that $\omega$ is not
	vertical with respect to the Iitaka fibration. Let $X$ be a minimal model with Iitaka fibration $f\colon X\rightarrow C$.
	Then there is a finite base change $\tl{C}\rightarrow C$ with induced fiber space $\tl{f}\colon \tl{X}\rightarrow \tl{C}$
	such that $\tl{X}\cong F\times \tl{C}$, where $F$ is abelian or K3.
\end{thm}
\begin{proof}[Proof of Proposition \ref{qqq}]
	By \cite[Theorem 1.6]{f2}, we have $h^1(X,\Oo_X)\leq h^1(F,\Oo_F)\leq 2$.
	Thus we have $h^1(X,\Oo_X)=2$ and $F$ should be an abelian surface. Let $a\colon X\rightarrow Alb(X)$
	be the Albanese map of $X$. We will prove that $a(F)$ is two-dimensional. \par
	Assume that $a(F)$ is a point, then $a(X)$ is a curve of genus $\geq2$. However this induces a morphism $C\rightarrow a(X)$,
	which is impossible. If $a(F)$ is one-dimensional. Then $a(F)$ should be an elliptic curve since otherwise there is a holomorphic map
	\[ F\rightarrow a(F)\rightarrow Jac(a(F))\]
	and the image should be a (translation of) non-trivial sub-complex torus, which is impossible.
	Since there are only countably many elliptic curves up to translation contained in a fixed abelian variety,
	we have $a(F)\cong E$ for some one-dimensional abelian subvariety $E$ of $Alb(X)$,
	for general $F$. By \cite[Theorem 5.3.5]{bl}, there is an isogeny $Alb(X)\rightarrow E\times T$ for some elliptic curve $T$.
	Consider the morphism \[X\rightarrow Alb(X)\rightarrow E\times T\rightarrow T\] which contracts the general fiber of $X\rightarrow C$.
	Thus there is an induced morphism $C\rightarrow T$. But $C\cong\Pp^1$ and $T$ is an elliptic curve. It leads a contradiction.
	In conclusion, we can prove that $a(F)$ is two-dimensional.\par
	Now we know that $a|_F\colon F\rightarrow Alb(X)$ is surjective. For any smooth model $W\rightarrow X$ we know that the pull-back
	of the global two-form of $Alb(X)$ on $W$ is a non-vertical two-form, hence $X\rightarrow C$ is isotrivial by Theorem \ref{cp}.
	We have to estimate $\deg A$. Consider the following diagram
	\[\xymatrix{F\times \tl{C} \ar[r]^-{\sim} & \tl{X}\ar[d]_{\tl{f}}\ar[r]^{\tl{\psi}} & X\ar[d]^f \\ &\tl{C}\ar[r]^{\psi} & C}.\]
	After replacing $\tl{C}$ by its Galois closure over $C$ we may assume that $\tl{C}\rightarrow C$ is Galois.
	Let $G=Gal(\tl{C}/C)$. Note that there is a natural $G$-action on $\tl{X}$
	such that $X\cong \tl{X}/G$. We have the induced action on $F$. 
	Since finite automorphism groups of an abelian surface are discrete, $G$ acts on
	$\tl{X}\cong F\times \tl{C}$ diagonally. We may assume further that $G$ acts on $F$ faithfully since the kernel of $G$ acting
	on $F$ is a normal subgroup of $G$ and one can replace $\tl{C}$ by $\tl{C}$ module the kernel.
	In particular $\tl{X}\rightarrow X$ is \'{e}tale in codimension one, hence $K_{\tl{X}}={\tl{\psi}}\st K_X$.\par	
	Let $B_0=\sum_{P\in C}(1-\frac{1}{u_P})P$. It is easy to see that $K_{\tl{C}}=\psi\st(K_C+B_0)$. Hence
	$K_{\tl{X}}={\tl{\psi}}\st f\st(K_C+B_0)$ which implies that $K_X=f\st(K_C+B_0)$. Since $B_0\leq B$ we have $M=0$ and
	$B=B_0=\sum_{P\in C}(1-\frac{1}{u_P})P$.\par
	Finally note that $\tl{C}$ is of general type and it is well-known that $\deg(K_C+B_0)\geq\frac{1}{42}$
	(by a simple calculation or using the fact that $|Aut(\tl{C})|\leq 84(g(\tl{C})-1)$). Hence $\deg A\geq \frac{1}{42}$.
\end{proof}
\subsection{Abelian or bielliptic fibrations with irregularity equal to one}
Assume that $h^1(X,\Oo_X)=1$.
Let $a\colon X\rightarrow E$ be the Albanese map of $X$. Let $D$ be a general fiber of $X\rightarrow E$.
Note that the induced morphism $f|_D\colon D\rightarrow C$
is surjective, since otherwise the Iitaka fibration $X\rightarrow C$ factors through $X\rightarrow E$, which is impossible.
Since $K_D=K_X|_D$, we know that $K_D$ is nef, $K_D\not\equiv0$ and $K^2_D=0$. Thus $D$ is a minimal surface of Kodaira dimension one and
the Iitaka fibration of $D$ factors through $D\rightarrow C$.
Consider the following diagram
\[\xymatrix{ D\ar[d]_{f_D}\ar[r]^{\iota} & X \ar[r]^a\ar[d]^f & E \\ C_0 \ar[r]^{\phi} & C & }.\]
where $f_D\colon D\rightarrow C_0$ is the Iitaka fibration of $D$. There is a $\Q$-divisor $A_D=K_{C_0}+M_D+B_D$ such that $f_D\st A_D=K_D$.\par
Recall that we write $K_X=f\st A=f\st(K_C+M+B)$. The $\Q$-divisor $A$ (resp. $A_D$) depends only on the embedding
$f\ts\Oo_X(mK_X)\hookrightarrow K(C)$ (resp. ${f_D}\ts\Oo_D(mK_D)\hookrightarrow K(C_0)$) for some fixed sufficiently divisible integer $m$. 
One may first fix an embedding ${f_D}\ts\Oo_D(mK_D)\hookrightarrow K(C_0)$ and then choose $f\ts\Oo_X(mK_X)\hookrightarrow K(C)$
via the canonical map
\[ f\ts\Oo_X(mK_X)\rightarrow f\ts\iota\ts\iota\st\Oo_X(mK_X)=f\ts\iota\ts\Oo_D(mK_D)=\phi\ts{f_D}\ts\Oo_D(mK_D)\rightarrow\phi\ts K(C_0)\]
(Note that the image of the above map is contained in $K(C)$). In conclusion, we may assume that $\phi\st A=A_D$.\par
\begin{lem}\label{lmb}
	For any $P'\in C_0$, let $m_{P'}$ be the multiplicity of the fibration $D\rightarrow C_0$ over $P'$. Let 
	\[ B_0=\sum_{P\in C}(1-\frac{1}{u_P})P,\quad B'_0=\sum_{P'\in C_0}(1-\frac{1}{m_{P'}})P',\] here $u_P$ is defined in Theorem \ref{fmcbf}.
	Then we have $\phi\st(K_C+B_0)=K_{C_0}+B'_0$.
\end{lem}
\begin{proof}
	Fix $P\in C$ and let $F_0$ be a reduced and irreducible component of $f^{-1}(P)$ such that $F_0\rightarrow E$ is surjective.
	Let $F_0\rightarrow E_0\rightarrow E$ be the Stein factorization of $F_0\rightarrow E$. For general $P$ we know that $F_0$ is either an
	abelian surface or a bielliptic surface, which implies $F_0\rightarrow E_0$ is smooth and $E_0\rightarrow E$ is \'{e}tale, hence
	$F_0\rightarrow E$ is smooth. Thus for a general point $Q\in E$ we may assume that $Q$ belongs to the smooth locus of $F_0\rightarrow E$,
	for every possible choice of $F_0$. In conclusion, we may assume that $D|_{F_0}=F_0|_D$ is reduced.\par
	This implies that if we fix $P\in C$ and write \[\phi\st P=\sum_ie_{P'_i}P'_i,\quad f_D\st P'_i=m_{P'_i}(f|_D\st P'_i)_{red},\]
	then $e_{P'_i}m_{P'_i}=e_{P'_j}m_{P'_j}$ for all $i$, $j$ and this number equals to $u_P$.\par
	We write $R=\sum_{P'\in C_0}(e_{P'}-1)P'$. We have
	\begin{align*}
		\phi\st(K_C+B_0)=K_{C_0}-R+\phi\st B_0&=K_{C_0}+\sum_{P'\in C_0}\left(1-e_{P'}+e_{P'}(1-\frac{1}{e_{P'}m_{P'}})\right)P'\\
		&=K_{C_0}+\sum_{P'\in C_0}\left(1-\frac{1}{m_{P'}}\right)P'=K_{C_0}+B'_0.
	\end{align*} 
\end{proof}
\begin{cor}\label{mbeff}
	Notation as above. Assume that $C_0\cong\Pp^1$. Let $B_1=B-B_0$. Then we may assume that $M+B_1$ is effective.
\end{cor}
\begin{proof}
	Let $B'_1=B_D-B'_0$, then we have $\phi\st(M+B_1)=M_D+B'_1$ since $\phi\st(K_C+B_0)=K_{C_0}+B'_0$. Because $M_D+B'_1$ is integral
	(cf. Lemma \ref{efcbf}), we may assume that $M_D+B'_1$ is effective since $C_0\cong\Pp^1$. Thus $M+B_1$ is effective.
\end{proof}
Recall that we have a morphism $J\colon C_0\rightarrow \Pp^1$ such that $12M_D=J\st\Oo_{\Pp^1}(1)$ (cf. Lemma \ref{efcbf}).
\begin{lem}
	Notation as above. If $J\colon C_0\rightarrow\Pp^1$ is a constant map, then $\deg A\geq\frac{1}{42}$. 
	Moreover, assume that $X\rightarrow C$ is an abelian fibration, then $h^2(X,\Oo_X)\geq1$ and $h^3(X,\Oo_X)=0$.
\end{lem}
\begin{proof}
	Note that $J\colon C_0\rightarrow\Pp^1$ is a constant map implies $M_D=0$. Furthermore, $B'_1$ corresponds to the contribution of
	non-multiple fibers in Kodaira's canonical bundle formula. Since $J\colon C_0\rightarrow\Pp^1$ is a constant map, all the singular fibers
	of $D\rightarrow C_0$ are multiple fibers. Thus $B'_1=0$ and we have $A_D=K_{C_0}+B'_0$ as well as $A=K_C+B_0$.
	So \[\deg A=-2+\sum_{P\in C}\left(1-\frac{1}{u_P}\right).\] It is easy to see that $\deg A\geq \frac{1}{42}$.\par
	Notice that we have $\rd{M+B}=0$. Assume that $f\colon X\rightarrow C$ is an abelian fibration, then
	$f\ts\Oo_X(K_X)=\Oo_C(\rd{A})=\Oo_C(-2)$. Thus $h^2(X,\Oo_X)=h^1(X,K_X)\geq h^1(C,f\ts\Oo_X(K_X))=1$ and
	$h^3(X,\Oo_X)=h^0(X,K_X)=h^0(C,f\ts\Oo_X(K_X))=0$.
\end{proof}
\begin{lem}
	Notation as above. Assume that $J\colon C_0\rightarrow\Pp^1$ is non-constant, then $\deg A\geq\frac{1}{72}$.
	Moreover, if $X\rightarrow C$ is an abelian fibration and $h^3(X,\Oo_X)>0$, then $\deg A\geq\frac{1}{24}$.
\end{lem}
\begin{proof}
	Note that since $F$ is either an abelian surface or a bielliptic surface, all the curves on $F$ which are contracted by
	$a|_F\colon F\rightarrow E$ are isomorphic to each other, hence all the components of $D\cap F$ are isomorphic.
	It implies that $J\colon C_0\rightarrow \Pp^1$ factors through
	$C_0\rightarrow C$. Let $n=\deg C_0/C$ and $n'$ be the degree of $J\colon C_0\rightarrow\Pp^1$, then we have $n'\geq n$.\par
	Since $12M_D=J\st\Oo_{\Pp^1}(1)$, we have $\deg M_D=\frac{n'}{12}$. Since $\deg M_D\leq\deg(M_D+B'_1)=\chi(\Oo_D)$ by Lemma \ref{efcbf},
	We have $n\leq n'\leq 12\chi(\Oo_D)$. Thus
	\begin{align*}
	\deg A=\frac{1}{n}\deg A_D
	&\geq \frac{1}{n}(2g(C_0)-2+\chi(\Oo_D))\geq\frac{2g(C_0)-2+\chi(\Oo_D)}{12\chi(\Oo_D)}.
	\end{align*}
	If $g(C_0)\geq 1$ we have $\deg A\geq\frac{1}{12}$. Now we assume that $g(C_0)=0$.\par
	Since $0<\deg M_D\leq\chi(\Oo_D)$, we have $\chi(\Oo_D)\geq 1$. When $\chi(\Oo_D)\geq3$, one can see that $\deg A\geq\frac{1}{36}$.
	If $\chi(\Oo_D)=2$ then $n\leq 24$ and the fibration $D\rightarrow C_0$ has at least one multiple fiber, hence $\deg A_D\geq\frac{1}{2}$.
	We have $\deg A\geq\frac{1}{48}$. If $\chi(\Oo_D)=1$ then $n\leq 12$ and
	\[\deg A_D=-1+\sum_i\left(1-\frac{1}{m_i}\right)\geq\frac{1}{6}.\]
	One has $\deg A\geq\frac{1}{72}$.\par
	Finally we assume that $X\rightarrow C$ is an abelian fibration and $h^3(X,\Oo_X)>0$ and we are going to prove that $\deg A\geq\frac{1}{24}$.
	If $h^3(X,\Oo_X)=h^0(X,K_X)>1$ then \[\deg A\geq\deg\rd{A}=\deg f\ts\Oo_X(K_X)\geq 1.\] Thus we may assume that $h^3(X,\Oo_X)=1$.\par
	Note that we may still assume that $g(C_0)=0$ since otherwise $\deg A\geq\frac{1}{12}$ as the computation above.
	The condition $h^3(X,\Oo_X)=h^0(X,K_X)=1$ implies \[-2+\deg\rd{M+B_0+B_1}=\deg\rd{A}=\deg f\ts\Oo_X(K_X)=0.\]
	One has $\deg\rd{M+B_0+B_1}=2$. Since $\rd{B_0}=0$, we have $M+B_1$ is supported on at least two points.
	Corollary \ref{mbeff} says that $M+B_1$ is effective. The coefficients of $M+B_1$ are greater than or equal to
	$\frac{1}{12}$ by the canonical bundle formula (see the discussion at the beginning of this section), hence $\deg(M+B_1)\geq\frac{1}{6}$.
	We have \[ \chi(\Oo_D)=\deg(M_D+B'_1)=n\deg(M+B_1)\geq\frac{n}{6},\] hence $n\leq 6\chi(\Oo_D)$ and
	\[\deg A=\frac{1}{n}\deg A_D\geq\frac{-2+\chi(\Oo_D)}{6\chi(\Oo_D)}.\]
	Also notice that $M_D\neq0$ implies $h^1(D,\Oo_D)=0$ because general fibers of $D\rightarrow C_0$ are not isomorphic to each other.
	Since $h^0(X,K_X)>0$, we have $h^0(D,K_D)=h^2(D,\Oo_D)>0$, thus $\chi(\Oo_D)\geq 2$.
	If $\chi(\Oo_D)\geq3$, then $\deg A\geq\frac{1}{18}$. Assume that $\chi(\Oo_D)=2$, then $n\leq12$. Also in this case $B'_0\neq0$,
	hence $\deg A'=\deg B'_0\geq\frac{1}{2}$ and $\deg A=\frac{1}{n}\deg A'\geq\frac{1}{24}$.
\end{proof}
Combining the previous two lemmas, one can conclude that
\begin{pro}\label{irr}
	Assume that $h^1(X,\Oo_X)=1$, then $\deg A\geq\frac{1}{72}$.
	Moreover, if $X\rightarrow C$ is an abelian fibration and $h^3(X,\Oo_X)>0$, then $\deg A\geq\frac{1}{24}$.
\end{pro}

\subsection{General situation}
\begin{lem}\label{gorr}
	Assume that $\chi(\Oo_X)=0$, then $\deg A\geq\frac{1}{156}$.
	Furthermore, if $h^2(X,\Oo_X)>0$ then either $h^3(X,\Oo_X)>0$ and $\deg A\geq \frac{1}{24}$, or
	$\deg M=0$, $B=\sum_{P\in C}(1-\frac{1}{u_P})P$ and $\deg A\geq\frac{1}{42}$.
\end{lem}
\begin{proof}
	We have $h^1(X,\Oo_X)+h^3(X,\Oo_X)>0$.
	If $h^1(X,\Oo_X)\neq0$ then $\deg A\geq\frac{1}{72}$ by Proposition \ref{qqq} and Proposition \ref{irr}.
	Now assume that $h^1(X,\Oo_X)=0$ and so $h^3(X,\Oo_X)>0$. In this case $h^0(X,K_X)>0$, hence $h^0(F,K_F)>0$ and $F$ is an abelian surface.
	Recall that Lemma \ref{cbf} implies $h^0(X,K_X)=\deg\rd{A}+1$. If $h^0(X,K_X)>1$, then $\deg A\geq 1$.
	Now we assume that $h^0(X,K_X)=1$. In this case we have $\deg \rd{A}=0$, hence $\deg A=\deg\{A\}$, where $\{A\}$ denotes
	the fractional part of $A$. By the canonical bundle formula we have
	$\{A\}=\sum_{P\in C}d_PP$ where $d_P=\frac{a_Pu_P-v_P}{n_Pu_P}$ for some positive integers $n_P\leq 12$,
	$a_P$, $u_P$ and $v_P$ such that $v_P\leq n_P$. Fix $P$ such that $d_P\neq0$. If $u_P\leq n_P$, then $d_P\geq\frac{1}{n_P^2}$.
	If $u_P>n_P$, then \[d_P\geq \frac{u_P-n_P}{n_Pu_P}=\frac{1}{n_P}-\frac{1}{u_P}\geq \frac{1}{n_P}-\frac{1}{n_P+1}
		\geq\frac{1}{12}-\frac{1}{13}=\frac{1}{156}.\]\par
	Finally assume that $h^2(X,\Oo_X)>0$, then $h^1(X,\Oo_X)+h^3(X,\Oo_X)\geq 2$.
	We have either $h^1(X,\Oo_X)>1$ which implies $\deg M=0$, $B=\sum_{P\in C}(1-\frac{1}{u_P})P$ and
	$\deg A\geq\frac{1}{42}$ by Proposition \ref{qqq}, or
	$h^1(X,\Oo_X)=1$, $h^3(X,\Oo_X)>0$ and $\deg A\geq\frac{1}{24}$ by Proposition \ref{irr}.
\end{proof}
\begin{cor}\label{gor}
	Assume that $X$ is Gorenstein, then $\deg A\geq\frac{1}{156}$ and $|mK_X|$ defines the Iitaka fibration if
	$m\geq 314$ and $|mK_F|$ is non-empty.
\end{cor}
\begin{proof}
	Since $X$ is Gorenstein, we have $\chi(\Oo_X)=0$ by Lemma \ref{chi}. Lemma \ref{gorr} says that $\deg A\geq\frac{1}{156}$ and hence
	$|mK_X|$ defines the Iitaka fibration if $m\geq 314$ and if $|mK_F|$ is non-empty by Lemma \ref{deg}.
\end{proof}

Let $\mu\colon \tl{X}\rightarrow X$ be the canonical cover (cf. Lemma \ref{cv}). We know that $\tl{X}$ is minimal of Kodaira dimension one.
Let $\tl{f}\colon \tl{X}\rightarrow \tl{C}$ be the Iitaka fibration of $\tl{X}$, and let $\nu\colon \tl{C}\rightarrow C$ be the induced morphism.
One may write $K_{\tl{X}}=\tl{f}\st\tl{A}=\tl{f}\st(K_{\tl{C}}+\tl{M}+\tl{B})$. Then $\tl{A}=\nu\st A$ and $\tl{M}=\nu\st M$. 

\begin{lem}\label{chi1}
	Assume that $\chi(\Oo_X)\leq 1$, then the global Cartier index of $K_X$ belongs to the following set:
	\[\{1,2,3,4,5,6,8,10,12\}.\]
\end{lem}
\begin{proof}
	By Lemma \ref{chi} we have $24\chi(\Oo_X)=\sum_{P\in\B(X)}\left(r_P-\frac{1}{r_P}\right)$.
	If $\chi(\Oo_X)=0$ then $X$ is Gorenstein. Assume that $\chi(\Oo_X)=1$. One can check that $\{r_P\}$ is one of the following set:
	\begin{enumerate}
	\item $\{r_{P_1}=...=r_{P_{16}}=2\}$.
	\item $\{r_{P_1}=...=r_{P_6}=2;r_{P_7}=...=r_{P_{10}}=4\}$.
	\item $\{r_{P_1}=...=r_{P_5}=2;r_{P_6}=...=r_{P_9}=3;r_{P_{10}}=6\}$.
	\item $\{r_{P_1}=...=r_{P_3}=2;r_{P_4}=4;r_{P_5}=r_{P_6}=8\}$.
	\item $\{r_{P_1}=...=r_{P_3}=2;r_{P_4}=r_{P_5}=5;r_{P_6}=10\}$.
	\item $\{r_{P_1}=r_{P_2}=2;r_{P_3}=r_{P_4}=3;r_{P_5}=4;r_{P_6}=12\}$.
	\item $\{r_{P_1}=...=r_{P_9}=3\}$.
	\item $\{r_{P_1}=...=r_{P_5}=5\}$.
	\end{enumerate}
\end{proof}
\begin{lem}\label{bie}
	Assume that $\tl{X}\rightarrow \tl{C}$ is a bielliptic fibration , then $\deg A\geq\frac{1}{864}$.
\end{lem}
\begin{proof}
	We know that $h^3(\tl{X},\Oo_{\tl{X}})=h^0(\tl{X},K_{\tl{X}})=0$ and $h^1(\tl{X},\Oo_{\tl{X}})\leq 1$. Since $\chi(\Oo_{\tl{X}})=0$,
	$h^1(\tl{X},\Oo_{\tl{X}})=1$ and $h^2(\tl{X},\Oo_{\tl{X}})=0$. Hence $h^2(X,\Oo_X)=0$ and so $\chi(\Oo_X)\leq 1$.
	Lemma \ref{chi1} says that the global Cartier index of $K_X$ is less than or equal to $12$, hence $\deg \tl{C}/C\leq 12$.
	Proposition \ref{irr} says that $\deg\tl{A}\geq \frac{1}{72}$. Since $\tl{A}=\nu\st A$,
	$\deg A\geq \frac{1}{72\cdot 12}=\frac{1}{864}$.
\end{proof}
From now on we may assume that $\tl{X}\rightarrow \tl{C}$ is an abelian fibration.
Recall that for any $P\in C$, we define $u_P$ be the smallest integer satisfying the following condition:
$\lambda f\st P$ is integral if and only if $\lambda\in\frac{1}{u_P}\Z$.
\begin{lem}
	For any $P\in C$ there is an integer $w_P$ such that $u_P$ divides $w_P$ and for any $Q\in\nu^{-1}(P)\subset\tl{C}$
	we have $e_Qu_Q=w_P$, where $e_Q$ is the ramification index of $Q$ with respect to $\nu\colon \tl{C}\rightarrow C$.
	Moreover, we have $\sum_{P\in C}(1-\frac{1}{w_P})P\leq B$.
\end{lem}
\begin{proof}
	Let $D=f^{-1}(P)$ and let $\mu^{-1}(D)=\bigcup_i D_i$ be the composition of connected components. Let $Q_i=\tl{f}(D_i)$, then
	$\{Q_i\}$ are all distinct and $\nu^{-1}(P)=\{Q_i\}$. Since there is a cyclic action on $\tl{X}$ permutes $D_i$, one may write
	$(\mu\circ f)\st P=\mu\st(u_PD)=\sum_i\lambda u_PD_i$ for some integer $\lambda$ which is independent of $i$. 
	We define $w_P=\lambda u_P$.\par	
	Note that $(\tl{f}\circ\nu)\st P=\tl{f}\st(\sum_i e_{Q_i}Q_i)=\sum_ie_{Q_i}u_{Q_i}D_i$, hence $e_{Q_i}u_{Q_i}=w_P$.
	Consider \begin{align*}
		K_{\tl{C}}+\sum_{Q\in\tl{C}}\left(1-\frac{1}{u_Q}\right)Q
			&=\nu\st K_C+\sum_{Q\in\tl{C}}(e_Q-1)Q+\sum_{Q\in\tl{C}}\left(1-\frac{1}{u_Q}\right)Q\\
		&=\nu\st K_C+\sum_{Q\in\tl{C}}e_Q\left(1-\frac{1}{w_P}\right) Q=\nu\st(K_C+\sum_{P\in C}\left(1-\frac{1}{w_P}\right)P).
	\end{align*}
	Note that $\sum_{Q\in\tl{C}}(1-\frac{1}{u_Q})Q\leq \tl{B}$ and $K_{\tl{C}}+\tl{B}=\nu\st(K_C+B)$ because of $\nu\st M=\tl{M}$.
	Hence $\sum_{P\in C}(1-\frac{1}{w_P})P\leq B$
\end{proof}
We define $B_0=\sum_{P\in C}(1-\frac{1}{w_P})P$ and $\tl{B}_0=\sum_{Q\in\tl{C}}(1-\frac{1}{u_Q})Q$. We have
$\nu\st(K_C+B_0)=K_{\tl{C}}+\tl{B}_0$. We define $B_1=B-B_0$ and $\tl{B}_1=\tl{B}-\tl{B_0}$, so that we also have $\nu\st B_1=\tl{B}_1$.
Recall that for $P\in C$ we define $n_P$ to be the smallest integer such that $n_PM$ is integral at $P$. 
\begin{lem}\label{b1}
	Assume that $\tl{X}\rightarrow \tl{C}$ is an abelian fibration.
	If $B_1=\sum_{P\in C}s'_PP$, then $s'_P\in\frac{1}{n_Pw_P}\Z_{\geq0}$.
\end{lem}
\begin{proof}
	For any $Q\in \nu^{-1}(P)$, we have $\mbox{coeff}_Q\tl{B}_1\in\frac{1}{n_Qu_Q}\Z_{\geq0}$. Since $\nu\st M=\tl{M}$,
	$n_Q$ divides $n_P$. Now $\nu\st B_1=\tl{B}_1$ implies $e_Qs'_P\in\frac{1}{n_Pu_Q}\Z_{\geq0}$,
	or $s'_P\in\frac{1}{n_Pu_Qe_Q}\Z_{\geq0}=\frac{1}{n_Pw_P}\Z_{\geq0}$.
\end{proof}
\begin{lem}\label{b0big}
	Assume that $\deg(K_C+B_0)\geq0$, then $\deg A\geq \frac{1}{120}$.
\end{lem}
\begin{proof}
	If $\deg(K_C+B_0)>0$, then $\deg(K_C+B_0)\geq\frac{1}{42}$. Hence \[\deg A=\deg(K_C+B+M)\geq\deg(K_C+B_0)\geq\frac{1}{42}.\]
	Assume that $\deg(K_C+B_0)=0$. First we assume that $B_1\neq0$. Note that $B_0$ is of one of the following form:
	\begin{enumerate}
	\item $\sum_{i=1}^4\frac{1}{2}P_i$.
	\item $\frac{1}{2}P_1+\frac{2}{3}P_2+\frac{5}{6}P_3$.
	\item $\frac{1}{2}P_1+\frac{3}{4}P_2+\frac{3}{4}P_3$.
	\item $\sum_{i=1}^3\frac{1}{3}P_i$.
	\end{enumerate}
	Hence $w_P\leq 6$ for all $P$. We have $n_P\leq 12$ for $\phi(n_P)\leq H^2_{prim}(F)=5$, thus
	\[\deg A=\deg(K_C+B_0+B_1+M)\geq\deg B_1\geq\frac{1}{72}\] by Lemma \ref{b1}.\par
	Finally assume that $B_1=0$. Then $\deg M>0$. Since $120M$ is integral, $\deg A=\deg M\geq\frac{1}{120}$.
\end{proof}
From now on we may assume that $\deg(K_C+B_0)<0$. There are only three possibilities of $B_0$:
\begin{enumerate}
	\item $B_0=\frac{m-1}{m}P+\frac{m'-1}{m'}P'$.
	\item $B_0=\frac{1}{2}P+\frac{1}{2}P'+\frac{m-1}{m}P''$.
	\item $B_0=\frac{1}{2}P+\frac{2}{3}P'+\frac{m-1}{m}P''$, $3\leq m\leq 5$.
\end{enumerate}
Note that $\deg(K_C+B_0)<0$ implies $\deg(K_{\tl{C}}+\tl{B}_0)<0$, hence $\tl{C}\cong\Pp^1$.
\begin{lem}\label{ty1}
	Assume that $B_0=\frac{m-1}{m}P+\frac{m'-1}{m'}P'$ and assume that $h^3(\tl{X},\Oo_{\tl{X}})\neq0$. Let
	$n=\deg \tl{C}/C$, then $\deg A\geq\frac{1}{60}-\frac{2}{n}$.
\end{lem}
\begin{proof}
	Note that $\deg(K_{\tl{C}}+\tl{B}_0)=n\deg(K_C+B_0)=-n(\frac{1}{m}+\frac{1}{m'})$. Since $\deg(K_{\tl{C}}+\tl{B}_0)\geq -2$,
	we have $\frac{1}{m}+\frac{1}{m'}\leq \frac{2}{n}$.\par
	\begin{labeling}{aa}
	\item[\textbf{Case 1}:] $Supp(B_1)\not\subset\{P,P'\}$. In this case there exists $P_1\in Supp(B_1)$ and $w_{P_1}=1$.
		Hence $\deg B_1\geq\frac{1}{12}$
		by Lemma \ref{b1}. So \[ \deg A\geq \deg(K_C+B_0+B_1)\geq \frac{1}{12}-\frac{1}{m}-\frac{1}{m'}\geq\frac{1}{12}-\frac{2}{n}.\]
	\item[\textbf{Case 2}:] $Supp(B_1)\subset\{P,P'\}$. In this case $Supp(B)=\{P,P'\}$, so $\deg B<2$.
		Since $\deg A=-2+\deg{M+B}>0$, we have $\deg M>0$. Consider the following two situations.
		\begin{labeling}{bb}
		\item[\textbf{Case 2-1}:] $Supp(\{M+B_1\})\subset\{P,P'\}$. We have $\#Supp(\{M\})\leq 2$ and so $\deg M\geq\frac{1}{60}$. Thus 
			\[ \deg A\geq \deg (K_C+B_0+M)\geq \frac{1}{60}-\frac{1}{m}-\frac{1}{m'}\geq\frac{1}{60}-\frac{2}{n}.\]
		\item[\textbf{Case 2-2}:] $Supp(\{M+B_1\})\not\subset\{P,P'\}$. If $\deg\rd{M+B_1}\geq0$, then 
			$\deg M+B_1=\deg\rd{M+B_1}+\deg\{M+B_1\}\geq\frac{1}{12}$, hence
			\[ \deg A\geq \deg(M+B_1)-\frac{1}{m}-\frac{1}{m'}\geq \frac{1}{12}-\frac{1}{m}-\frac{1}{m'}\geq\frac{1}{12}-\frac{2}{n}.\]\par
			Assume that $\deg\rd{M+B_1}<0$. 
			Since $\deg(K_{\tl{C}}+\tl{B}_0)\leq0$, we have $\#Supp(\tl{B}_0)\leq 3$. Note that $h^3(\tl{X},\Oo_{\tl{X}})\neq0$ implies
			$\deg \rd{\tl{M}+\tl{B}}=\deg\rd{\tl{M}+\tl{B}_0+\tl{B}_1}\geq 2$, hence
			$\deg \rd{\tl{M}+\tl{B}_1}\geq 2-\#Supp(\tl{B}_0)$. Write $\{M+B_1\}=aP+a'P'+\sum_ib_iP_i$ and assume that $\deg\rd{M+B_1}=-d$.
			Our assumption implies $b_i\neq0$ for some $i$.
			We have \[\deg\rd{\tl{M}+\tl{B}_1}=\sum_i \rd{ae_{Q_i}}+\sum_j\rd{a'e_{Q'_j}}-nd\geq 2-\#Supp(\tl{B}_0),\]
			where $e_Q$ denotes the ramification index with respect to $\tl{C}\rightarrow C$. In other words,
			\[ a+a'\geq \frac{1}{n}\left(\sum_i \rd{ae_{Q_i}}+\sum_j\rd{a'e_{Q'_j}}\right)\geq d+\frac{2-\#Supp(\tl{B}_0)}{n}\]
			since $n=\sum_ie_{Q_i}=\sum_je_{Q'_j}$. This implies
			\[\deg(M+B_1)\geq \sum_ib_i+\frac{2-\#Supp(\tl{B}_0)}{n}\geq \frac{1}{12}+\frac{2-\#Supp(\tl{B}_0)}{n}.\]
			Note that if $\#Supp(\tl{B}_0)=3$, then $\deg(K_{\tl{C}}+\tl{B}_0)>-1$ and then $\frac{1}{m}+\frac{1}{m'}<\frac{1}{n}$.
			We conclude that \[\frac{2-\#Supp(\tl{B}_0)}{n}-\frac{1}{m}-\frac{1}{m'}\geq -\frac{2}{n},\] so
			\[ \deg A\geq \deg(M+B_1)-\frac{1}{m}-\frac{1}{m'}\geq \frac{1}{12}+\frac{2-\#Supp(\tl{B}_0)}{n}-\frac{1}{m}-\frac{1}{m'}
				\geq\frac{1}{12}-\frac{2}{n}.\]
		\end{labeling}
	\end{labeling}		
\end{proof}
\begin{lem}\label{ty2}
	Assume that $B_0=\frac{1}{2}P+\frac{1}{2}P'+\frac{m-1}{m}P''$ and assume that $h^3(\tl{X},\Oo_{\tl{X}})\neq0$. Let
	$n=\deg \tl{C}/C$. If $n\geq 6$ then $\deg A\geq\frac{1}{36}-\frac{2}{n}$.
\end{lem}
\begin{proof}
	We use the same argument as in the previous lemma, but this case are more complicated.
	We have \[-2\leq\deg(K_{\tl{C}}+\tl{B}_0)=n\deg(K_C+B_0)=-\frac{n}{m},\] hence
	$\frac{1}{m}\leq\frac{2}{n}$.\par
	\begin{labeling}{a}
	\item[\textbf{Case 1}:]	$Supp(B_1)\not\subset\{P''\}$. By Lemma \ref{b1} if we write $B_1=\sum_{P_i\in C}s_iP_i$, then
		$s_i\in\frac{1}{n_{P_i}w_{P_i}}\Z_{\geq0}$. We know that there exists $P_i\neq P''$ such that $s_i\neq0$.
		Since $w_{P_i}\leq 2$ if $P_i\neq P''$, $s_i\geq\frac{1}{24}$. Thus $\deg B_1\geq \frac{1}{24}$ and
		\[\deg A\geq\deg(K_C+B_0+B_1)\geq\frac{1}{24}-\frac{1}{m}\geq \frac{1}{24}-\frac{2}{n}.\]
	\item[\textbf{Case 2}:]	$Supp(B_1)=\{P''\}$. In this case $B=\frac{1}{2}P+\frac{1}{2}P'+s_{P''}P''$, so $\deg B<2$ and we have $\deg M>0$
		(because $\deg K_C+M+B>0$).
		\begin{labeling}{a}
		\item[\textbf{Case 2-1}:] $Supp(\{M+B_1\})=\{P''\}$. We have $Supp(\{M\})=P''$ and $\deg M\geq\frac{1}{12}$.
			so \[\deg A\geq\deg(K_C+B_0+B_1+M)\geq\frac{1}{12}-\frac{1}{m}\geq \frac{1}{12}-\frac{2}{n}.\]
		\item[\textbf{Case 2-2}:] $Supp(\{M+B_1\})\neq\{P''\}$.
			\begin{labeling}{a}
			\item[\textbf{Case 2-2-1}:] $\deg\rd{M+B_1}\geq0$. We have $\deg(M+B_1)\geq\deg(\{M+B_1\})\geq\frac{1}{12}$ and so
				\[\deg A\geq\deg(K_C+B_0+B_1+M)\geq\frac{1}{12}-\frac{1}{m}\geq \frac{1}{12}-\frac{2}{n}.\]
			\item[\textbf{Case 2-2-2}:] $\deg\rd{M+B_1}=-d<0$. We write $\{M+B_1\}=aP+a'P'+a''P''+\sum_ib_iP_i$ and
				we have $a$, $a'$ and $b_i$ are not all zero. The condition $h^3(\tl{X},\Oo_{\tl{X}})\neq0$ implies
				$\deg\rd{\tl{M}+\tl{B}_1+\tl{B}_0}\geq 2$ and so $\deg \rd{\tl{M}+\tl{B}_1}\geq 2-\#Supp(\tl{B}_0)$.
				Now we have \[\deg\rd{\tl{M}+\tl{B}_1}=
					\sum_i \rd{ae_{Q_i}}+\sum_j\rd{a'e_{Q'_j}}+\sum_k\rd{a''e_{Q''_k}}-nd\geq 2-\#Supp(\tl{B}_0),\]
				where $e_Q$ denotes the ramification index with respect to $\tl{C}\rightarrow C$. Hence
				\begin{align*}
				a+a'+a''&\geq\frac{1}{n}\left(\sum_i(\rd{ae_{Q_i}}+\{ae_{Q_i}\})+\sum_j(\rd{a'e_{Q'_j}}+\{a'e_{Q'_j}\})
					+\sum_k\rd{a''e_{Q''_k}}\right)\\
				&\geq \frac{1}{n}\left(\sum_i\{ae_{Q_i}\}+\sum_j\{ae_{Q'_j}\}\right)+
				\frac{1}{n}\left(\sum_i \rd{ae_{Q_i}}Q_i+\sum_j\rd{a'e_{Q'_j}}Q'_j+\sum_k\rd{a''e_{Q''_k}}\right)\\
				&\geq \frac{1}{n}\left(\sum_i\{ae_{Q_i}\}+\sum_j\{a'e_{Q'_j}\}\right)+d+\frac{2-\#Supp(\tl{B}_0)}{n}.
				\end{align*}
				\begin{labeling}{a}
				\item[\textbf{Case 2-2-2-1}:] $b_i\neq0$ for some $i$. In this case $\deg(M+B_1)\geq b_i+\frac{2-\#Supp(\tl{B}_0)}{n}$.
					Since $b_i\geq\frac{1}{12}$, we have
					\[\deg A\geq\deg(K_C+B_0+B_1+M)\geq\frac{1}{12}-\frac{1}{m}+\frac{2-\#Supp(\tl{B}_0)}{n}.\]
				\item[\textbf{Case 2-2-2-2}:] $b_i=0$ for all $i$. Then $a$ or $a'\neq0$. Observe the following fact:
					we have $e_{Q_i}$ and $e_{Q'_j}\leq 2$ and
					$e_{Q_i}=1$ if and only if $Q_i\in Supp(\tl{B}_0)$. Let $k=\#\se{Q_i}{e_{Q_i}=1}$ and $k'=\#\se{Q'_j}{e_{Q'_j}=1}$, then
					\[n=2(\#\{\nu^{-1}(P)\}-k)+k=2(\#\{\nu^{-1}(P')\}-k')+k'.\] Hence $k-k'$ is an even integer. Since
					$\#Supp(\tl{B}_0)\leq 3$, we have $k$ and $k'\leq 2$. Assume that $n\geq 6$, then $k\leq\frac{n}{3}$ and
					we have $n\leq 3(\#\{\nu^{-1}(P)\}-k)$. This implies
					\[\frac{1}{n}\sum_i\{ae_{Q_i}\}=\frac{\#\{\nu^{-1}(P)\}-k}{n}\{2a\}\geq\frac{1}{3}\{2a\}.\]
					The same argument yields $\frac{1}{n}\sum_j\{a'e_{Q'_j}\}\geq\frac{1}{3}\{2a'\}$.
					Now we have
					\[\deg(M+B_1)=a+a'+a''-d\geq\frac{1}{3}(\{2a\}+\{2a'\})+\frac{2-\#Supp(\tl{B}_0)}{n}.\]
					If one of $\{2a\}$ and $\{2a'\}$ is non-zero, then it is greater than or equal to $\frac{1}{12}$. Hence we have
					$\deg(M+B_1)\geq \frac{1}{36}+\frac{2-\#Supp(\tl{B}_0)}{n}$.
					Thus
					\[ \deg A\geq \deg(M+B_1)-\frac{1}{m}\geq \frac{1}{36}+\frac{2-\#Supp(\tl{B}_0)}{n}-\frac{1}{m}\geq\frac{1}{36}-\frac{2}{n}\]
					by noticing that $\#Supp(\tl{B}_0)>2$ implies $\frac{1}{m}\leq\frac{1}{n}$.\par
					Finally if both $\{2a\}$ and $\{2a'\}$ is zero, then $2M$ is integral outside $P''$. This implies either $10M$ or $12M$ is
					integral. Hence $\deg M\geq \frac{1}{12}$ and we have
					\[\deg A\geq\deg(K_C+B_0+M)\geq\frac{1}{12}-\frac{1}{m}\geq \frac{1}{12}-\frac{2}{n}.\]
				\end{labeling} 
			\end{labeling} 
		\end{labeling} 
	\end{labeling}
\end{proof}
\begin{pro}\label{nGor}
	Assume that $X$ is not Gorenstein, then $\deg A\geq \frac{1}{2928}$. Hence $|mK_X|$ defines the Iitaka fibration if
	$m\geq 5858$ and $|mK_F|$ is non-empty.\par
\end{pro}
\begin{proof}
	If $\tl{X}\rightarrow \tl{C}$ is a bielliptic fibration, then $\deg A\geq \frac{1}{864}$ by Lemma \ref{bie}.
	Now assume that $\tl{X}\rightarrow \tl{C}$ is an abelian fibration. If $\deg(K_X+B_0)\geq 0$, then $\deg A\geq\frac{1}{120}$ by
	Lemma \ref{b0big}. Assume that $\deg(K_X+B_0)<0$. If $\chi(\Oo_X)=1$ then the global Cartier index of $K_X$ is less than
	or equal to $12$ by Lemma \ref{chi1}, hence the degree of $\tl{X}\rightarrow X$ is less than or equal to $12$. This implies
	$\deg(\tl{C}/C)\leq 12$. Since $\deg\tl{A}\leq\frac{1}{156}$ and $\tl{A}=\nu\st A$, we have
	\[\deg A\leq\frac{1}{12\cdot156}=\frac{1}{1872}.\]
	From now on we assume that $\chi(\Oo_X)>1$, which implies $h^2(X,\Oo_X)>0$. Hence $h^2(\tl{X},\Oo_{\tl{X}})>0$.
	Recall that $\chi(\Oo_{\tl{X}})=0$ since $\tl{X}$ is Gorenstein. Lemma \ref{gorr} says that either
	$\deg \tl{M}=0$, $\tl{B}=\sum_{P\in C}(1-\frac{1}{u_P})P$ or $\deg\tl{A}\geq\frac{1}{24}$. 
	In the former case we have $\tl{B}=\tl{B}_0$. One has $\deg\tl{A}=\deg(K_{\tl{C}}+\tl{B}_0)$ and $\deg A=\deg(K_C+B_0)$.
	Thus $\deg A\geq\frac{1}{42}$. Now assume that $\deg\tl{A}\geq\frac{1}{24}$. Let $n=\deg(\tl{C}/C)$.
	We know that $B_0$ is of one of the following form:
	\begin{enumerate}
		\item $B_0=\frac{m-1}{m}P+\frac{m'-1}{m'}P'$.
		\item $B_0=\frac{1}{2}P+\frac{1}{2}P'+\frac{m-1}{m}P''$.
		\item $B_0=\frac{1}{2}P+\frac{2}{3}P'+\frac{m-1}{m}P''$, $3\leq m\leq 5$.
	\end{enumerate}
	If $B_0=\frac{1}{2}P+\frac{2}{3}P'+\frac{m-1}{m}P''$ for $3\leq m\leq 5$ then $\deg(K_C+B_0)=-\frac{1}{6}$, $-\frac{1}{12}$ or
	$-\frac{1}{30}$. Since $-2\leq\deg(K_{\tl{C}}+\tl{B}_0)=n\deg(K_C+B_0)$, we have $n\leq 60$. Since $\deg\tl{A}\geq\frac{1}{24}$,
	$\deg A\geq \frac{1}{1440}$.\par
	Now assume that $B_0=\frac{m-1}{m}P+\frac{m'-1}{m'}P'$. If $n\leq 122$ we have \[\deg A\geq\frac{1}{24\cdot122}=\frac{1}{2928}.\]
	Assume that $n\geq123$. Lemma \ref{ty1} implies \[\deg A\geq\frac{1}{60}-\frac{2}{n}\geq\frac{1}{60}-\frac{2}{123}=\frac{1}{2460}.\]
	Finally assume that $B_0=\frac{1}{2}P+\frac{1}{2}P'+\frac{m-1}{m}P''$. If $n\leq 73$ we have
	\[\deg A\geq\frac{1}{24\cdot73}=\frac{1}{1752}.\] If $n\geq 74$ then Lemma \ref{ty2} implies 
	\[\deg A\geq\frac{1}{36}-\frac{2}{n}\geq\frac{1}{36}-\frac{1}{37}=\frac{1}{1332}.\]
\end{proof}

\section{Boundedness of Iitaka fibration for Kodaira dimension one}\label{sk1}
\begin{proof}[Proof of Theorem \ref{mt}]
	If $C$ is not rational, this follows from Proposition \ref{nonr}.
	The K3 or Enriques fibration cases follow from Proposition \ref{k3} and the abelian or bielliptic fibration cases follow from
	Corollary \ref{gor} and Proposition \ref{nGor}.
\end{proof}
In the rest part of this section we will compute several examples.
\begin{eg}
	The first example is a trivial example. Let $F$ be a bielliptic curve such that $|6K_F|$ is non-empty but $|iK_F|$ is empty
	for all $i\leq5$ and let $C$ be a curve of general type. Then $X=F\times C$ is a smooth threefold of Kodaira dimension one
	such that $|6K_X|$ defines the Iitaka fibration but $|iK_X|$ is empty for all $i\leq5$.
\end{eg}
\begin{eg}
	Let $E$ be an elliptic curve. Pick two different points $P$ and $Q$ on $E$. One can find a line bundle $L$ such that $L^2=\Oo_E(P+Q)$.
	Let $C$ be the cyclic cover corresponds to $L^2=\Oo_E(P+Q)$. Then $C$ is a curve of genus two and $\phi\colon C\rightarrow E$ is a double cover
	ramified at $P$ and $Q$. Let $G=Aut(C/E)$, which is a cyclic group of order two. Let $F$ be an abelian surface and consider the $G$-action
	on $F$ via $-\mathbf{Id}$. Let $X=(F\times C)/G$.\par
	The singular points of $X$ are of the type $\frac{1}{2}(1,1,1)$, hence $X$ has terminal singularities. We want to show that $|4K_X|$ 
	defines the Iitaka fibration, and $|iK_X|$ does not define the Iitaka fibration for $i\leq3$. One has
	\[H^0(X,mK_X)=H^0(F\times C,mK_F\boxtimes mK_C)^G=H^0(C,mK_C)^G\]
	since the unique section in $H^0(F,mK_F)$ is fixed by $G$ for all $m$. To compute $H^0(C,mK_C)^G$, note that
	$\phi\ts\Oo_C=\Oo_E\oplus L^{-1}$ and $\Oo_C(2K_C)=\phi\st\Oo_E(2K_E+P+Q)=\phi\st L^2$, hence
	$\phi\ts\Oo_C(2kK_C)= L^{2k}\oplus L^{2k-1}$ by the projection formula.
	The $G$-invariant part of $H^0(C,2kK_C)$ is $H^0(E,L^{2k})$ and $L^{2k}$ is very ample if and only if $k\geq2$.
	Hence $|2K_X|$ does not define the Iitaka fibration, but $|4K_X|$ does.\par
	On the other hand, by Grothendieck duality we have 
	\begin{align*}
		\phi\ts \Oo_C((2k+1)K_C)&=\phi\ts Hom_{\Oo_C}(\Oo_C(-2kK_C),\Oo_C(K_C))\\
		&\cong Hom_{\Oo_E}(\phi\ts\Oo_C((-2kK_C)),\Oo_E(K_E))\cong L^{2k}\oplus L^{2k+1}.
	\end{align*}		
	This shows that $h^0(C,3K_C)^G=h^0(E,L^2)=2$ and hence $|3K_X|$ do not define the Iitaka fibration.\par
	We remark that this is the worst example we know for abelian fibrations. 
\end{eg}
\begin{eg}\label{e42}
	Let $C$ be the Klein quartic \[(x^3y+y^3z+z^3x=0)\subset\Pp^2.\] It is known that \[|G|=|Aut(C)|=168=42(2g(C)-2)\]
	(c.f. \cite[Section 6.5.3]{d}). Let \[F=(x^3y+y^3z+z^3x+u^4=0)\subset\Pp^3,\] which is a K3 surface. Define the $G$-action on $F$ by
	$g([x\colon y\colon z\colon u])=[g([x\colon y\colon z])\colon u]$. Let $X=(F\times C)/G$. We will prove the following:
	\begin{enumerate}[(1)]
	\item $X$ has terminal singularities.
	\item $H^0(X,iK_X)\leq1$ for $i\leq41$ and $H^0(X,42K_X)=2$.
	\end{enumerate}
	Hence the smooth model of $X$ is a threefold of Kodaira dimension one, such that $|42K_X|$ defines the Iitaka fibration, but
	$|iK_X|$ do not define the Iitaka fibration for $i\leq 41$.\par
	First we prove (1). Since $|Aut(C)|=168=42(2g(C)-2)$, it is well-known that the morphism $C\rightarrow C/G$ ramified at three points
	$P_2$, $P_3$ and $P_7\in C/G$ and the
	stabilizer of points over $P_r$ is a cyclic group of order $r$ for $r=2$, $3$ and $7$ (cf. also \cite[Proposition in Section 2.1]{e}).
	Let $F_r\subset X$ be the fiber over $P_r$. We need to compute the singularities of $F_r$.\par
	\begin{enumerate}[(i)]
	\item $r=7$. Note that any order $7$ subgroup of $G$ is a Sylow-subgroup, which is unique up to conjugation. To compute the singularities
		we may assume that the stabilizer is the cyclic group generated by the element
		(see \cite[Section 1.1]{e} for the description of elements in $G$)
		\[\sigma=\begin{pmatrix}\xi^4 & 0 & 0 \\ 0 & \xi^2 & 0 \\ 0 & 0 & \xi \end{pmatrix},\quad \xi=e^{\frac{2\pi i}{7}}.\]
		One can compute that $H_7=\langle \sigma\rangle$ has three fixed point $[1\colon 0\colon 0\colon 0]$, $[0\colon 1\colon 0\colon 0]$
		and $[0\colon 0\colon 1\colon 0]$ on $F_7$. The $H_7$ action
		around those fixed points is of the form $\frac{1}{7}(4,3)$, $\frac{1}{7}(2,5)$ and $\frac{1}{7}(1,6)$ respectively.
		The conclusion is that $X$ has three singular points over $P_7$ which are cyclic quotient points of the form
		$\frac{1}{7}(1,6,1)$, $\frac{1}{7}(2,5,1)$ and $\frac{1}{7}(3,4,1)$ respectively.
	\item $r=3$. As before any order $3$ subgroup of $G$ is a Sylow-subgroup and hence we may assume that the stabilizer is generated by
		\[\tau=\begin{pmatrix}0&1&0\\0&0&1\\1&0&0\end{pmatrix}.\]
		There are six fixed points on $F_3$, namely $[1\colon \omega\colon \omega^2\colon 0]$, $[1\colon \omega^2\colon \omega\colon 0]$ and
		$[1\colon 1\colon 1\colon u_i]_{i=1,...,4}$, where $w=e^{\frac{2\pi i}{3}}$ and $u_1$, ..., $u_4$ are four roots of the equation
		$u^4+3=0$. Let $P=[1\colon \omega\colon \omega^2\colon 0]$. The local coordinates near $P\in\A^3$ are
		$y'=y/x-\omega$, $z'=z/x-\omega^2$ and $u'=u/x$. Let $\alpha=y'+z'$ and $\beta=\omega y'+\omega^2 z'$.
		Then $\alpha$, $\beta$ and $u'$ are also local coordinates near $P$ and the defining equation of $F_3$ near $P$ can be written as
		\[ (1+3\omega)\alpha+\mbox{ higher order terms},\] hence the local coordinate of $P\in F_3$ is $\beta$ and $u'$.\par
		We have \[\tau(u')=\tau(\frac{u}{x})=\frac{u}{z}=\frac{1}{z'+\omega^2}u'=\omega\phi u',\]
		where $\phi$ is a holomorphic function satisfying $\phi(P)=1$ and $\phi\tau(\phi)\tau^2(\phi)=1$.
		Let $\lambda$ be a holomorphic function near $P$ such that $\lambda^3=\phi$ and $\lambda(P)=1$. One can check that
		\[\tau(\lambda' u')=\omega \lambda' u',\quad \mbox{where }\lambda'=\lambda^2\tau(\lambda).\]\par
		On the other hand, 
		\[ \tau(\beta)=\frac{\omega}{z'+\omega^2}\beta=\omega^2\phi\beta\] and we also have
		\[\tau(\lambda'\beta)=\omega^2\lambda'\beta.\] The conclusion is that the singularity of $P\in F_3$ is of the form $\frac{1}{3}(1,2)$,
		and hence the singularity of $P\in X$ is a terminal cyclic quotient $\frac{1}{3}(1,2,1)$.
		A similar computation (simply interchange $\omega$ and $\omega^2$ in the calculation) shows that
		$P'=[1\colon \omega^2\colon \omega\colon 0]\in F_3\subset X$
		is also a terminal cyclic quotient point.\par
		Finally we compute the singularities of $Q_i=[1\colon 1\colon 1\colon u_i]$ for $i=1$, ..., $4$. Let $y_0=y/x-1$ and $z_0=z/x-1$.
		We take $\alpha_0=\omega y_0+\omega^2 z_0$ and $\beta_0=\omega^2 y_0+\omega z_0$ as local coordinates of $Q_i\in F_3$. One can see that
		\[ \tau(\alpha_0)=\tau(\omega y_0+\omega^2 z_0)=\tau(\omega \frac{y}{x}+\omega^2 \frac{z}{x}+1)
			=\frac{x}{z}(\omega+\omega^2 \frac{y}{x}+\frac{z}{x})=\frac{\omega}{z_0+1}\alpha_0\]
		and \[\tau(\beta_0)=\frac{\omega^2}{z_0+1} \beta_0.\]
		Using the same technique above we can say that the singularity of $Q_i\in F_3$ is of the form $\frac{1}{3}(1,2)$,
		hence $Q_i\in X$ is also a terminal cyclic quotient point for $i=1$, ..., $4$.
	\item $r=2$. Let $\mu\in G$ be an order two element. We have to compute the singularities of $F_2/\langle \mu\rangle$. By
		\cite[Proposition in Section 2.1]{e}, $\mu$ fixes a line and a point in $\Pp^2$. By the character table of $G$
		(cf. \cite[Section 1.1]{e}), we know that the three-dimensional character of $\mu$ is equal to $-1$.
		This implies the fixed line of $\mu$ in $\Pp^2$ corresponds to the two-dimensional eigenspace with eigenvalue $-1$,
		and the fixed point of $\mu$ in $\Pp^2$ corresponds to the one-dimensional eigenspace with
		eigenvalue $1$. Assume that $L$ is the fixed line of $\mu$ in $\Pp^2$, $L\cap C=[x_i\colon y_i\colon z_i]_{i=1,...,4}$
		and the fixed point of $\mu$ in $\Pp^2$ is $[x_5\colon y_5\colon z_5]$. One can check that the fixed point of $\mu$ on $F_2$ is
		$[x_i\colon y_i\colon z_i\colon 0]$ for $i=1$, ..., $4$ and $[x_5\colon y_5\colon z_5\colon u_j]_{j=1,...,4}$,
		where $u_j$ are the roots of the equation $u^4+x_5^3y_5+y_5^3z_5+z_5^3x_5=0$. The conclusion is that there are eight
		cyclic quotient points of index two on $F_2$. Since they are isolated, all the singular points should be the from $\frac{1}{2}(1,1)$.
		The conclusion is that there are eight singular points on $X$ which is of the from $\frac{1}{2}(1,1,1)$.
	\end{enumerate}\par
	Note that the Iitaka fibration of $X$ is a K3 fibration, and the basket data of $X$ is
	\[ \{(2,1)\times8,(3,1)\times6,(7,1),(7,2),(7,3)\}.\] It is the worst case in Section \ref{sk3}.\par
	Now we prove (2). We need to compute $H^0(X,mK_X)=H^0(F\times C,mK_F\boxtimes mK_C)^G$. Consider the long exact sequence
	\begin{align*}
		0\rightarrow &H^0(\Pp^3,mK_{\Pp^3}+(m-1)F)\rightarrow H^0(\Pp^3,m(K_{\Pp^3}+F))\rightarrow H^0(F,mK_F)\rightarrow \\
		&H^1(\Pp^3,mK_{\Pp^3}+(m-1)F)\rightarrow\cdots.
	\end{align*}	 
	Since $H^i(\Pp^3,mK_{\Pp^3}+(m-1)F)=0$ for $i=0$, $1$, we have \[H^0(F,mK_F)=H^0(\Pp^3,m(K_{\Pp^3}+F))=H^0(\Pp^3,\Oo_{\Pp^3}).\]
	Thus any section in $H^0(F,mK_F)$ is $G$-invariant. This tell us that
	\[H^0(X,mK_X)=H^0(F\times C,mK_F\boxtimes mK_C)^G=H^0(C,mK_C)^G.\]\par
	One can consider the following long exact sequence
	\begin{align*}
		0\rightarrow& H^0(\Pp^2,mK_{\Pp^2}+(m-1)C)\rightarrow H^0(\Pp^2,m(K_{\Pp^2}+C))\rightarrow H^0(C,mK_C)\rightarrow \\
		&H^1(\Pp^2,mK_{\Pp^2}+(m-1)C)\rightarrow\cdots.
	\end{align*}
	Since $H^1(\Pp^2,mK_{\Pp^2}+(m-1)C)=0$, the restriction map
	\[ H^0(\Pp^2,m(K_{\Pp^2}+C))=H^0(\Pp^2,\Oo_{\Pp^2}(m))\rightarrow H^0(C,mK_C)\] is surjective.
	Thus to find $G$-invariant sections in $H^0(C,mK_C)$ is equivalent to find $G$-invariant polynomials of degree $m$ on $C$.
	It is known that (c.f. \cite[Section 1.2]{e}) the $G$-invariant polynomials are generated by three elements $f_6$, $f_{14}$ and $f_{21}$,
	where $f_d$ is a polynomial of degree $d$, satisfying $f_{21}^2=f_{14}^3-1728f_6^7$. Hence $h^0(C,iK_C)^G\leq1$ for all $i\leq41$
	and $H^0(C,42K_C)^G$ is spanned by $f_6^7$ and $f_{14}^3$. Thus $h^0(X,iK_X)\leq1$ for $i\leq41$ and $h^0(X,42K_X)=2$. 
	
\end{eg}

\small{
Department of Mathematics, National Taiwan University, No. 1, Sec. 4, Roosevelt Rd., Taipei 10617, Taiwan\\
\it{Email address}: d02221002@ntu.edu.tw}
\appendix
\section{The algorithm computing basket data}
We give the algorithm we used in the proof of Proposition \ref{k3}, about finding possible basket data for a K3 or an Enriques fibration.\par
Recall for our notation: $X$ is a minimal terminal threefold of Kodaira dimension one, $f\colon X\rightarrow C$ is the Iitaka fibration and
a general fiber $F$ of the fibration is either a K3 or an Enriques surface.
\begin{description}
\item[Input data]$ $
	\begin{enumerate}[(i)]
		\item $\chi(\Oo_F)$. This number is either $2$ (when $F$ is a K3 surface) or $1$ (if $F$ is an Enriques surface).		
		\item $\chi(\Oo_X)$. When $F$ is an Enriques surface, $\chi(\Oo_X)=1$. If $F$ is a K3 surface, we may assume that
			$0\leq\chi(\Oo_X)\leq 2$ by the argument in the first page of Section \ref{sk3}.
		\item An positive integer $N$.
	\end{enumerate} 
\item[Output data] A set $S_{\chi(\Oo_F),\chi(\Oo_X),N}$ of the basket data $\{(r_i,b_i)\}_{i\in I}$ such that $2b_i\leq r_i$ and
	$gcd(r_i,b_i)=1$ for all $i\in I$.
\item[Conditions] For all $\{(r_i,b_i)\}_{i\in I}\in S_{\chi(\Oo_F),\chi(\Oo_X),N}$, we have
	\begin{enumerate}[(1)]
		\item \[24\chi(\Oo_X)<\sum_{i\in I}\left(r_i-\frac{1}{r_i}\right)<24\chi(\Oo_X)+\frac{12}{N}\chi(\Oo_F).\]
		\item For all $m=2$, ..., $lcm\{r_i\}_{i\in I}$, we have
			\[(1-2m)\chi(\Oo_X)+\sum_{i\in I}\sum_{j=1}^{m-1}\frac{\overline{jb_i}(r_i-\overline{jb_i})}{2r_i}\]
			is a non-negative integer.
	\end{enumerate}
\end{description}\par
If we write $K_X=f\st A$ for some ample $\Q$-divisor $A$ on $C$ and assume that $\deg A<\frac{1}{N}$, then the basket data of $X$ must belongs
to $S_{\chi(\Oo_F),\chi(\Oo_X),N}$. One can choose a sufficiently large $N$, so that $S_{\chi(\Oo_F),\chi(\Oo_X),N}$ do not contain
too many elements. Then one can compute $\deg A$ using Lemma \ref{chi} case by case.
\begin{description}
\item[Result]$ $
	\begin{enumerate}[(1)]
	\item $F$ is a K3 surface:
		\begin{enumerate}[({1}-1)]
		\item $\chi(\Oo_X)=0$: We take $N=2$.
			\[ S_{2,0,2}=\left\lbrace
			\begin{tabular}{l}
			$\se{(2,1)\times 4}{\deg A=\frac{1}{4}}$\\
			$\se{(2,1),(3,1),(6,1)}{\deg A=\frac{5}{12}}$\\
			$\se{(2,1),(4,1)\times 2}{\deg A=\frac{3}{8}}$\\
			$\se{(3,1)\times 3}{\deg A=\frac{1}{3}}$\\
			$\se{(5,1),(5,2)}{\deg A=\frac{2}{5}}$
			\end{tabular} \right\rbrace.\]
		\item $\chi(\Oo_X)=1$: We take $N=12$.
			\[ S_{2,1,12}=\left\lbrace
			\begin{tabular}{l}
				$\se{(2,1)\times 6,(3,1),(4,1)\times 2,(6,1)}{\deg A=\frac{1}{24}}$\\
				$\se{(2,1)\times 5,(3,1),(5,1),(5,2),(6,1)}{\deg A=\frac{1}{15}}$\\
				$\se{(2,1)\times 5,(4,1)\times 2,(5,1),(5,2)}{\deg A=\frac{1}{40}}$\\
				$\se{(2,1)\times 4,(5,1)\times 2,(5,2)\times 2}{\deg A=\frac{1}{20}}$\\
				$\se{(2,1),(3,1)\times 6,(4,1)\times 2}{\deg A=\frac{1}{24}}$\\
				$\se{(3,1)\times 6,(5,1),(5,2)}{\deg A=\frac{1}{15}}$\\
				$\se{(13,4),(13,6)}{\deg A=\frac{1}{13}}$
			\end{tabular} \right\rbrace.\]
		\item $\chi(\Oo_X)=2$: We take $N=24$.
			\[ S_{2,2,24}=\left\lbrace 
				\begin{tabular}{l}
				$\se{(2,1)\times 8, (3,1)\times 6, (7,1), (7,2), (7,3)}{\deg A=\frac{1}{42}}$\\
			 	$\se{(2,1)\times 8, (3,1), (5,1), (5,2)\times 3, (15,4)}{\deg A=\frac{1}{30}}$ \\
			 	$\se{(2,1)\times 5, (4,1)\times 2, (17,6), (17,7)}{\deg A=\frac{5}{136}}$ \\
			 	$\se{(2,1)\times 5, (7,2)\times 2, (7,3)\times 2, (14,5)}{\deg A=\frac{1}{28}}$ \\
			 	$\se{(2,1)\times 4, (3,1)\times 2, (5,1), (5,2), (7,3), (21,8)}{\deg A=\frac{13}{420}}$ \\
			 	$\se{(2,1)\times 4, (3,1)\times 2, (5,1), (5,2), (8,3)\times 2, (12,5)}{\deg A=\frac{1}{40}}$ \\
			 	$\se{(2,1)\times 3, (3,1)\times 3, (5,1), (5,2), (9,4), (18,7)}{\deg A=\frac{7}{180}}$
				\end{tabular} \right\rbrace.\]
		\end{enumerate}
	\item $F$ is an Enriques surface: We take $N=6$. Note that the set $S_{1,1,6}$ is exactly $S_{2,1,12}$, but the corresponding 
		$\deg A$ are different.
		\[ S_{1,1,6}=\left\lbrace
			\begin{tabular}{l}
			$\se{(2,1)\times 6,(3,1),(4,1)\times 2,(6,1)}{\deg A=\frac{1}{12}}$\\
			$\se{(2,1)\times 5,(3,1),(5,1),(5,2),(6,1)}{\deg A=\frac{2}{15}}$\\
			$\se{(2,1)\times 5,(4,1)\times 2,(5,1),(5,2)}{\deg A=\frac{1}{20}}$\\
			$\se{(2,1)\times 4,(5,1)\times 2,(5,2)\times 2}{\deg A=\frac{1}{10}}$\\
			$\se{(2,1),(3,1)\times 6,(4,1)\times 2}{\deg A=\frac{1}{12}}$\\
			$\se{(3,1)\times 6,(5,1),(5,2)}{\deg A=\frac{2}{15}}$\\
			$\se{(13,4),(13,6)}{\deg A=\frac{2}{13}}$
			\end{tabular} \right\rbrace.\]
	\end{enumerate}
	The basket data with minimal degree is \[\{(2,1)\times 8, (3,1)\times 6, (7,1), (7,2), (7,3)\}\in S_{2,2,24}\]
	with $\deg A=\frac{1}{42}$. This basket data appears in Example \ref{e42}.
\item[Algorithm]$ $
	\begin{algorithm}
	\SetKwFunction{fr}{FindR}
	\SetKwFunction{fb}{FindB}
	\KwData{$\chi(\Oo_F)$, $\chi(\Oo_X)$ and $N$}
	\KwResult{A set $S_{\chi(\Oo_F),\chi(\Oo_X),N}$ consists of $\{(r_i,b_i)\}$ satisfying the condition (1) (2) above.}
	Let $T$ be an empty set\;
	\For{$r=2$ \KwTo $\rd{24\chi(\Oo_X)+\frac{12}{N}\chi(\Oo_F)}$}
		{ Execute \fr($\{r\}$) } \tcc*[l]{Find all possible $\{r_i\}$ satisfying the Condition (1); Possible solutions will be stored in $T$} 
	\ForEach{$R=\{r_i\}\in T$}
	{
		Execute \fb{R}, Let $B_R$ be the outcome of $\fb{R}$\tcc*[l]{$B_R$ consists of the set $\{b_i\}$ such that $(r_i,b_i)$ satisfies
			condition (2)}
	}
	Output $S_{\chi(\Oo_F),\chi(\Oo_X),N}=\se{(r_i,b_i)}{R=\{r_i\}\in T; \{b_i\}\in B_R}$\;
	\caption{Main process}
	\end{algorithm}	
	\begin{function}
	\SetKwInOut{ind}{Input Data}
	\SetKwInOut{goa}{Goal}
	\SetKwProg{Fn}{Function}{}{}
	\Fn{\fr{$R$}}
	{
		\ind{$R$ is a set of nature numbers}
		\goa{Find all possible set $\{r_i\}$ satisfying the Condition (1), with $R\subsetneq\{r_i\}$}
		Compute $r=\sum_{r_i\in R}\left(r_i-\frac{1}{r_i}\right)$\;
		Let $k=\max\se{r_i}{r_i\in R}$\;
		\While{$(r+k-\frac{1}{k})<24\chi(\Oo_X)+\frac{12}{N}\chi(\Oo_F)$}
		{
			\If{$(r+k-\frac{1}{k})>24\chi(\Oo_X)$}
			{
				Put the element $R\cup\{k\}$ in $T$\;	
			}
			Execute \fr($R\cup\{k\}$)\;
			Replace $k$ by $k+1$\;
			
		}
	}
	\caption{FindR()}
	\end{function}
	\begin{function}
	\SetKwProg{Fn}{Function}{}{}
	\SetKwInOut{ind}{Input Data}
	\SetKwInOut{goa}{Goal}
	\SetKwFunction{fn}{Function1}
	\Fn{\fb{$R$}}
	{
		\ind{$R=\{r_i\}$ is a set of nature numbers}
		\goa{Find all possible $\{b_i\}$ such that $\{(r_i,b_i)\}$ satisfies Condition (2)}
		Let $E_i=\se{a\in\N}{2a\leq r_i,gcd(a,r_i)=1}$\;
		Let $B=E_1$\;
		Execute \fn{$B$,$2$}\tcc*[l]{\fn find every pair $\{(r_i,b_i)\}$ such that $2b_i\leq r_i$ and $gcd(r_i,b_i)=1$ for all $i$,
			and store such data in $B$}
		\ForEach{$\{b_i\}\in B$}
		{
			\For{$m=2$ \KwTo $lcm\{r_i\}$}
			{			
				\If{$(r_i,b_i)$ do not satisfy Condition (2)}{Remove $\{b_i\}$ from $B$}
			}	
		}
		Output $B$;
	}
	\Fn{\fn{$B$,$j$}}
	{
		\ind{$B$ contains the data $\{b_i\}_{i=1}^{j-1}$ satisfying $2b_i\leq r_i$ and $gcd(r_i,b_i)=1$ for all $i\leq j-1$}
		\goa{Replace $B$ by $B'=\se{\{b_i\}}{2b_i\leq r_i\mbox{ and }gcd(r_i,b_i)=1\mbox{ for all }i}$}
		Let $B'$ be an empty set\;
		\ForEach{$\{b_i\}_{i=1}^{j-1}\in B$}
		{
			\ForEach{$e\in E_j$}
			{
				\If{$r_j\neq r_{j-1}$ {\bf or} ($r_j=r_{j-1}$ {\bf and} $e\geq b_{j-1}$)}
				{
					Put $\{b_i\}\cup \{e\}$ in $B'$\;
				}	
			}
		}
		Replace $B$ by $B'$\;
		\If{$j<|R|$}{ Execute \fn($B$,$j+1$)\;}
	}
	\caption{FindB()}
	\end{function}

\end{description}

\end{document}